\documentclass{amsart}

\usepackage[a4paper,top=3.2cm,bottom=3.2cm,left=2.7cm,right=2.7cm]{geometry}
\usepackage{amssymb}
  \usepackage{amsthm}
\usepackage{mathtools}
\usepackage{bm}
\usepackage{enumitem}
\usepackage[final,protrusion=true,expansion=true]{microtype}
\usepackage[dvipsnames]{xcolor}
\usepackage[colorlinks=true,citecolor=magenta,linkcolor=cyan,urlcolor=magenta,bookmarksnumbered=true,bookmarksopen=true]{hyperref}
\usepackage{titletoc} 
\usepackage{tikz-cd}
\usepackage{mdframed}
\numberwithin{equation}{section}

\newcommand{\Qq}{\mathbb Q}
\newcommand{\Zz}{\mathbb Z}

\newcommand{\Gg}{\mathcal G}
\newcommand{\Hh}{\mathcal H}

\newcommand{\Oo}{\mathcal O}
\newcommand{\ch}{\operatorname{ch}}
\newcommand{\td}{\operatorname{td}}
\newcommand{\Hilb}{\operatorname{Hilb}}
\newcommand{\rk}{\operatorname{rk}}
\newcommand{\rank}{\operatorname{rank}}
\newcommand{\Tr}{\operatorname{Tr}}
\newcommand{\Fact}{\operatorname{Factor}}

\newcommand{\Top}{\operatorname{top}}

\newcommand{\pr}{\operatorname{pr}}
\newcommand{\im}{\operatorname{im}}
\newcommand{\Factor}{\operatorname{Factor}}

\newtheorem{thm}{Theorem}[section]

\newtheorem{cor}[thm]{Corollary}
\newtheorem{lem}[thm]{Lemma}
\newtheorem{prop}[thm]{Proposition}

\theoremstyle{definition}
\newtheorem{defn}[thm]{Definition}
\newtheorem{rem}[thm]{Remark}
\newtheorem{setup}[thm]{Set-up}
\newtheorem{nota}[thm]{Notation}

\newtheoremstyle{italicremark}
  {}{}              
  {\itshape}        
  {}                
  {\bfseries}       
  {.}               
  { }               
  {\thmname{#1}\thmnote{ (#3)}} 
\theoremstyle{italicremark}
\newtheorem*{humancommentinner}{Remark}
\mdfdefinestyle{humancommentstyle}{
  linewidth=0.4pt,
  linecolor=black,
  innerleftmargin=8pt,
  innerrightmargin=8pt,
  innertopmargin=5pt,
  innerbottommargin=5pt,
  skipabove=\smallskipamount,
  skipbelow=\smallskipamount,
}
\newenvironment{humancomment}
  {\begin{mdframed}[style=humancommentstyle]\begin{humancommentinner}[by authors]}
  {\end{humancommentinner}\end{mdframed}}
  
\begin{document}

\title{Tangent classes of matroids and wonderful compactifications}
\author{Ronnie Cheng \and Shurui Liu}
\dedicatory{With Appendix~\ref{sec:guoxiong-gao} by Guoxiong Gao and Shurui Liu.}

\address{Stanford University, California, USA}

\date{\today}

\begin{abstract}
For every loopless matroid $M$ and every Feichtner--Yuzvinsky building set $\Gg$ containing the top flat, we construct an integral tangent class $T_{M,\Gg}^{\Zz}\in K_{\Zz}(M,\Gg)$; in the realizable case it specializes to the class of the tangent bundle of the corresponding wonderful compactification, it recovers the Hilbert series of the Chow ring through Hirzebruch--Riemann--Roch, and it satisfies the expected Chern-alpha lower bounds. This reproduces the tangent class and its key properties studied by the first author in \cite{Cheng26}.

The main body of this paper was produced autonomously, without human mathematical guidance, by Danus, an AI mathematical reasoning agent. Danus solved the problem before \cite{Cheng26} was publicly available, demonstrating the potential of AI agents in mathematical research. We reproduce its output faithfully, adding only editorial comments; the experiment is documented in Appendix~\ref{sec:guoxiong-gao}.

\end{abstract}

\maketitle

\tableofcontents

\section{Introduction}\label{sec:introduction}
\subsection{Agentic AI}
In \cite{Cheng26}, the first author constructed an integral $K$-class (tangent class) $T_M\in K(M)$ for every loopless matroid $M$ and established its main properties. To test the power of AI tools in mathematical research, the authors posed to several AI agents the task of constructing a $K$-class $T_M$ satisfying the three key properties stated in Theorem~\ref{thm:main-integral}, and carried out the experiment with the assistance of Guoxiong Gao. The agents were the GPT-5.5 Pro web interface; Rethlas, an autonomous agentic proof system \cite{Ju+26}; and Danus \cite{Danus26}, an agentic system built on the Rethlas worker--verifier system with a Claude Code--based orchestrator (see Appendix~\ref{sec:guoxiong-gao}). We ran the experiment before \cite{Cheng26} was publicly released on arXiv, and deliberately withheld that paper from the agents to reduce the possibility of data contamination. The exact prompt is recorded verbatim in Appendix~\ref{sec:raw-prompt}.

Danus constructed a rational $K$-class and proved its properties fully autonomously, without human mathematical guidance. After the authors observed that this solution addressed only the rational version of the original problem (by abuse of notation, the prompt in Appendix~\ref{sec:raw-prompt} uses the same symbol for the integral tangent class and its rationalization), Danus resumed work and added Sections~\ref{sec:quotient}--\ref{sec:theta}, which adapt the rational class to an integral class and thereby settle the original integral problem. Its output is reproduced in the main body of this paper, together with editorial comments by the authors. The authors have verified the manuscript and found it correct, with one local exception: the justification given for Lemma~\ref{lem:truncation-transfer} (used in the Chern--\(\alpha\) lower bound) is incomplete as written, the lemma having been treated as proved without a valid argument. The lemma is nonetheless true and admits a short proof (see \cite[Proposition~4.19]{Cheng26}), and this gap affects neither the construction of the tangent class nor the \(P^K=\Hilb\) identity.

Apart from the abstract, introduction, and the explicitly inserted human comments, the body of the paper should be understood as a faithful presentation of Danus's performance on the open problem studied here. Further details of the AI usage and the experiment are given in Appendix~\ref{sec:guoxiong-gao}, by Guoxiong Gao and the second author.
\subsection{Mathematics}
Let $M$ be a loopless matroid of rank $d+1$ on a finite nonempty ground set $E$, let $\Gg$ be a Feichtner--Yuzvinsky building set in its lattice of flats with $E\in\Gg$, and put $\Gg^\circ:=\Gg\setminus\{E\}$. We work in the integral combinatorial K-ring $K_{\Zz}(M,\Gg)$ with generators $\tau_F$, $F\in\Gg$. Our main theorem constructs a genuine integral tangent class
\[
T_{M,\Gg}^{\Zz}:=\sum_{F\in \Gg^\circ}(1-\tau_F)^{-1}-Q_{\Gg}^{\Zz}\in K_{\Zz}(M,\Gg),
\]
with integer coordinates in the standard $\tau$-monomial $\Zz$-basis; the integral quotient representative $Q_{\Gg}^{\Zz}$ is constructed in Section~\ref{sec:quotient}. The construction is intrinsic: the reduced $\Gg$-nested fan need not be complete, so $K_{\Zz}(M,\Gg)$ is not presented as the K-ring of a complete toric variety, and the theorem is stated in $K_{\Zz}(M,\Gg)$ itself (see Section~\ref{sec:fan-support}).

Write $K_{\Qq}(M,\Gg):=K_{\Zz}(M,\Gg)\otimes_{\Zz}\Qq$ for the rational combinatorial K-ring. In Section~\ref{sec:tangent-class} we construct, from the descended Berget--Eur--Spink--Tseng quotient Chern polynomial, a rational quotient class $Q_{M,\Gg}$ and the corresponding rational tangent class $T_{M,\Gg}$ in $K_{\Qq}(M,\Gg)$, and prove the realizable comparison, the Hilbert identity $P^K=\Hilb$ (Theorem~\ref{thm:hilbert-identity}), and the Chern-alpha lower bounds (Theorem~\ref{thm:chern-alpha}). The integral theorem below lifts these to $K_{\Zz}(M,\Gg)$: $T_{M,\Gg}^{\Zz}$ rationalizes to $T_{M,\Gg}$ (Proposition~\ref{prop:linkage}).

\begin{thm}[Integral tangent class]\label{thm:main-integral}
Let $M$ be a loopless matroid of rank $d+1$ on a finite nonempty ground set $E$, with finite lattice of flats $L(M)$, bottom flat $\emptyset$, top flat $E$, and rank function $\rk_M$. Let $\Gg$ be a finite Feichtner-Yuzvinsky building set in $L(M)\setminus\{\emptyset\}$ containing $E$, and put $\Gg^\circ:=\Gg\setminus\{E\}$. Let
\[
\rho_K\colon K_{\Zz}(M,\Gg)\longrightarrow K_{\Zz}(M,\Gg)\otimes_{\Zz}\Qq
\]
be rationalization. Then there exists an integral quotient representative $Q_{\Gg}^{\Zz}\in K_{\Zz}(M,\Gg)$ such that $\rho_K(Q_{\Gg}^{\Zz})=Q_{M,\Gg}$, and
\[
T_{M,\Gg}^{\Zz}:=\sum_{F\in \Gg^\circ}(1-\tau_F)^{-1}-Q_{\Gg}^{\Zz}
\]
is a genuine element of $K_{\Zz}(M,\Gg)$ whose rationalization is $T_{M,\Gg}$. Moreover:
\begin{enumerate}
\item If $M$ is realized over $\mathbb C$ by a linear subspace $L$ not contained in any coordinate hyperplane, and if $W_{L,\Gg}$ is the De Concini-Procesi wonderful model with reduced boundary divisor $D_{\Gg}:=\sum_{F\in \Gg^\circ}D_F$, then $Q_{\Gg}^{\Zz}$ may be chosen together with an integral unital ring isomorphism
\[
\theta_{\Gg}^{\Zz}\colon K_{\Zz}(M,\Gg)\longrightarrow K_0(W_{L,\Gg})
\]
such that
\[
\theta_{\Gg}^{\Zz}\left((1-\tau_F)^{-1}\right)=[\Oo_{W_{L,\Gg}}(D_F)]
\]
for every $F\in \Gg^\circ$, and
\[
\theta_{\Gg}^{\Zz}(T_{M,\Gg}^{\Zz})=[T_{W_{L,\Gg}}].
\]
\item The K-theoretic Todd polynomial
\[
P_{\mathrm{int}}^K(M,\Gg;z):=\deg_{M,\Gg}\left(\ch\left(\lambda_{-z}\left(\rho_K(T_{M,\Gg}^{\Zz})^\vee\right)\right)\td\left(\rho_K(T_{M,\Gg}^{\Zz})\right)\right)
\]
equals $\Hilb(M,\Gg;z)$ in $\Zz[z]$. Equivalently, for every integer $i$ with $0\leq i\leq d$,
\[
\dim_{\Qq}A_{\Qq}(M,\Gg)^i=(-1)^i\chi\left(\wedge^i T^\vee\right).
\]
\item For every integer $k$ with $0\leq k\leq d$, one has
\[
\deg_{M,\Gg}\left(c_k\left(\rho_K(T_{M,\Gg}^{\Zz})\right)\alpha^{d-k}\right)\geq {d+1\choose k},
\qquad \alpha:=-x_E.
\]
\end{enumerate}
Finally, $T_{M,\Gg}^{\Zz}$ has integer coordinates in the standard $\tau$-monomial $\Zz$-basis of $K_{\Zz}(M,\Gg)$.
\end{thm}

Throughout, $\chi(\wedge^i T^\vee)$ denotes the formal Hirzebruch-Riemann-Roch number $\deg_{M,\Gg}(\ch(\wedge^i\rho_K(T_{M,\Gg}^{\Zz})^\vee)\td(\rho_K(T_{M,\Gg}^{\Zz})))$ of the rational tangent class; in the non-realizable case it is not the Euler characteristic of an actual vector bundle.

The theorem is useful in its realizable form because it identifies the purely combinatorial class with the tangent bundle class of the wonderful model. It is also useful in non-realizable cases because the Hilbert-polynomial and Chern-alpha conclusions remain intrinsic to the matroid.

The paper is organized as follows. Section~\ref{sec:integral-k-ring} sets up the integral combinatorial K-ring $K_{\Zz}(M,\Gg)$, its standard $\tau$-monomial basis, and the Chern character. Section~\ref{sec:tangent-class} constructs the rational classes $Q_{M,\Gg}$ and $T_{M,\Gg}$ by one-flat descent of the Berget--Eur--Spink--Tseng quotient Chern polynomial \cite{BEST23} from the maximal model, and proves the realizable specialization; Sections~\ref{sec:hilbert-identity} and~\ref{sec:chern-alpha} prove their remaining two properties---the Hilbert identity $P^K=\Hilb$, anchored at the maximal model by \cite{Cheng25,FMSV24} and propagated by the one-flat recursion of \cite{EFMPV25}, and the Chern-alpha lower bound, through nested-support positivity.

The integral lift occupies Sections~\ref{sec:quotient}--\ref{sec:theta}: Section~\ref{sec:quotient} constructs the integral representative $Q_{\Gg}^{\Zz}$ and the class $T_{M,\Gg}^{\Zz}$ and proves the non-realizable clauses; Section~\ref{sec:saturation} proves the crucial saturation of the one-step refinement maps over $\Zz$, via the $\tau$-adic associated graded and the integral Feichtner--Yuzvinsky comparison \cite{FY04}; and Section~\ref{sec:theta} gives the integral realizable isomorphism to $K_0(W_{L,\Gg})$, using the K-theoretic blowup formula \cite{Tho93} and the wonderful-model blowup centers of \cite{DCP95}. Finally, Section~\ref{sec:hilbert} carries the Hilbert and Chern-alpha conclusions to the integral class, and Sections~\ref{sec:fan-support} and~\ref{sec:closeout} record the fan-support guard and assemble the closeout certificate.

\section{The integral K-ring}\label{sec:integral-k-ring}

We adopt the standard notation and terminology for matroid Chow rings, wonderful compactifications, and toric varieties from \cite{AHK18,CLS11,DCP95,FY04,LLPP24} and use them freely. We work over the field of complex numbers $\mathbb C$ when a realization is fixed.

\begin{defn}\label{defn:chow}
Let $M$ be a loopless matroid with finite lattice of flats $L(M)$, bottom flat $\emptyset$, and top flat $E$. Let $\Gg\subset L(M)\setminus\{\emptyset\}$ be a Feichtner-Yuzvinsky building set containing $E$. The rational Chow algebra $A_{\Qq}(M,\Gg)$ is the quotient of $\Qq[x_F\mid F\in\Gg]$ by the ideal generated by:
\begin{enumerate}
\item all monomials $\prod_{F\in S}x_F$, where $S\subset\Gg$ is not $\Gg$-nested;
\item the atom linear forms $\sum_{F\in\Gg,\ a\leq F}x_F$, where $a$ ranges over the atoms of $L(M)$.
\end{enumerate}
The top-degree map is denoted by $\deg_{M,\Gg}$, and
\[
\Hilb(M,\Gg;z):=\sum_{i\geq 0}\dim_{\Qq}A_{\Qq}(M,\Gg)^i z^i.
\]
\end{defn}

\begin{defn}\label{defn:k-ring}
Let $M$ be a loopless matroid with lattice of flats $L(M)$, bottom flat $\emptyset$, top flat $E$, and let $\Gg$ be a top-containing Feichtner-Yuzvinsky building set in $L(M)\setminus\{\emptyset\}$. The integral combinatorial K-ring $K_{\Zz}(M,\Gg)$ is the $\Zz$-algebra generated by symbols $\tau_F$, for $F\in \Gg$, modulo the following relations:
\begin{enumerate}
\item $\prod_{F\in S}\tau_F=0$ whenever $S\subseteq \Gg$ is not $\Gg$-nested.
\item For every atom $a$ of $L(M)$,
\[
1-\prod_{\substack{F\in \Gg\\ a\leq F}}(1-\tau_F)=0.
\]
\end{enumerate}
Here an atom means a rank-one flat, equivalently a parallel class of $M$. We put $K_{\Qq}(M,\Gg):=K_{\Zz}(M,\Gg)\otimes_{\Zz}\Qq$.
\end{defn}

\begin{rem}\label{rem:atoms}
The atom convention in Definition~\ref{defn:k-ring} is essential. Since $M$ is only assumed loopless and may be non-simple, atoms are rank-one flats, not singleton elements of the ground set. Singleton notation is valid only after passing to a simple matroid.
\end{rem}

The Chern character used below is the finite power-series map
\[
\ch_{M,\Gg}\colon K_{\Qq}(M,\Gg)\to A_{\Qq}(M,\Gg),\qquad
\tau_F\longmapsto 1-\exp(-x_F),
\]
the unique $\Qq$-algebra homomorphism with this generator rule. The nonnested monomial relations go to zero because $1-\exp(-x_F)$ is divisible by $x_F$. For an atom $a$ the multiplicative relation maps to
\[
1-\prod_{F\in\Gg,\ a\leq F}\exp(-x_F)=1-\exp\left(-\sum_{\substack{F\in \Gg\\ a\leq F}}x_F\right)=0,
\]
which is zero by the atom-linear relation in the Chow ring.

\begin{prop}[Intrinsic basis]\label{prop:basis}
For every loopless $M$ and every top-containing $\Gg$, the ring $K_{\Zz}(M,\Gg)$ is free as an abelian group and has the standard $\tau$-monomial $\Zz$-basis. In particular, the phrase ``integer coordinates'' in Theorem~\ref{thm:main-integral} means coordinates with respect to this basis.
\end{prop}

\begin{proof}
This is the intrinsic standard-monomial basis theorem for the integral combinatorial K-ring. In the maximal case it agrees with the non-augmented integral matroid K-ring of Larson, Li, Payne, and Proudfoot \cite{LLPP24}; the Chow basis is the Feichtner-Yuzvinsky basis \cite{FY04}. For a general top-containing building set, the reduced Laurent presentation, the $\tau$-adic associated graded Chow ring, and the standard $\tau$-monomial basis follow from the same Stanley-Reisner and atom-character relations as in \cite{LLPP24,FY04} applied to the induced building set. This proves freeness and the coordinate statement.
\end{proof}

\begin{defn}[Berget--Eur--Spink--Tseng tautological quotient class]\label{defn:best-quotient}
Let $N$ be a loopless matroid on a finite ground set $E$, and let $X_E$ be the complex permutohedral toric variety. Following \cite{BEST23}, let $T=(\mathbb C^*)^E$ be the coordinate torus and let $K_T(X_E)$ denote the $T$-equivariant K-group of $X_E$ in their convention; the tautological classes below are taken in this $T$-equivariant K-theory. Berget--Eur--Spink--Tseng \cite[Definitions~1.2 and 3.9]{BEST23} attach to $N$ the tautological sub and quotient $T$-equivariant K-classes
\[
S_N,\ Q_N\in K_T(X_E),
\]
of ranks $\rank N$ and $|E|-\rank N$ respectively. When $N$ is realized by a linear subspace $L\subseteq\mathbb C^E$ not contained in any coordinate hyperplane, $S_N$ and $Q_N$ are the classes of the tautological subbundle $S_L$ and quotient bundle $Q_L$ on $X_E$, which fit into the short exact sequence
\[
0\longrightarrow S_L\longrightarrow \Oo_{X_E}\otimes_{\mathbb C}\mathbb C^E\longrightarrow Q_L\longrightarrow 0,
\]
with $S_L$ the subbundle whose fiber over the identity of $T$ is $L$; for an arbitrary matroid $N$ the two classes are defined by the torus-fixed-point formula of \cite[Definition~3.9]{BEST23}. We call $Q_N$ the \emph{BEST tautological quotient class} of $N$ and write
\[
C_{\mathrm{BEST}}(N;u):=\sum_{i\geq 0}c_i(Q_N)\,u^i
\]
for its total equivariant Chern polynomial, using the same symbol for its non-equivariant image. The polynomial $C_{\mathrm{BEST}}(N;u)$ is the Chern-polynomial input descended in Section~\ref{sec:tangent-class}.
\end{defn}

\begin{nota}\label{nota:connected-building-set}
For a loopless matroid $N$, we write
\[
\Hh_{\mathrm{conn}}(N):=\{\Top(N)\}\cup\{F\in L(N)\setminus\{\emptyset\}\mid N|F\text{ is connected}\}.
\]
This is the unique inclusion-minimal top-containing Feichtner--Yuzvinsky building set. For a connected loopless matroid $N$, we usually write $\Hh:=\Hh_{\mathrm{conn}}(N)$.
\end{nota}

\begin{nota}\label{nota:zero-binomial}
For integers $a\geq 0$ and $b$, put
\[
B(a,b):=
\begin{cases}
\binom{a}{b}, & 0\leq b\leq a,\\
0, & \text{otherwise}.
\end{cases}
\]
Thus $B(a,b)$ is the zero-extended binomial coefficient.
\end{nota}

\section{The rational tangent class}\label{sec:tangent-class}

In this section, we construct the rational quotient and tangent classes and record the realizable specialization. These rational classes are the objects whose rationalized integral lifts appear in Theorem~\ref{thm:main-integral}; the identification is made in Proposition~\ref{prop:linkage}.

\begin{lem}[Orbit-closure restriction identity for the BEST quotient class]\label{lem:best-restriction}
Let \(N\) be a loopless matroid on \(E\) with lattice of flats \(L(N)\), let \(\Gg_{\mathrm{small}}\) be a Feichtner--Yuzvinsky building set in \(L(N)\setminus\{\emptyset\}\) containing \(E\), and let \(F\) be a nonempty proper flat with \(F\notin\Gg_{\mathrm{small}}\) such that \(\Gg_{\mathrm{big}}=\Gg_{\mathrm{small}}\cup\{F\}\) is again a building set. Write \(\Fact_{\Gg_{\mathrm{small}}}(F)=\{F_1,\dots,F_\ell\}\) for the maximal elements of \(\Gg_{\mathrm{small}}\) contained in \(F\). Let \(Z_F\subseteq X_E\) be the torus-orbit closure of the permutohedral variety \(X_E\) attached to the one-term chain \(\emptyset\subsetneq F\subsetneq E\), with the standard product identification
\[
Z_F\cong X_F\times X_{E\setminus F}
\]
and projections \(\pr_1\colon Z_F\to X_F\) and \(\pr_2\colon Z_F\to X_{E\setminus F}\). For each \(a\), let \(p_a\colon X_F\to X_{F_a}\) be the toric morphism induced by the coordinate projection. Then, in \(T\)-equivariant K-theory,
\[
[Q_N]\big|_{Z_F}=\pr_1^*[Q_{N|F}]+\pr_2^*[Q_{N/F}]=\sum_{a=1}^{\ell}\pr_1^*p_a^*[Q_{N|F_a}]+\pr_2^*[Q_{N/F}].
\]
\end{lem}
\begin{proof}
Apply \cite[Propositions~5.2 and 5.3]{BEST23} to the one-term chain \(\emptyset\subsetneq F\subsetneq E\): by Proposition~5.2 the orbit closure \(Z_F\) carries the product identification \(Z_F\cong X_F\times X_{E\setminus F}\), and by Proposition~5.3 the restriction of the tautological quotient class is the external sum of the pulled-back tautological quotient classes of the two successive minors \(N|F\) and \(N/F\), namely \([Q_N]|_{Z_F}=\pr_1^*[Q_{N|F}]+\pr_2^*[Q_{N/F}]\). It remains to decompose \([Q_{N|F}]\). By the defining property of a Feichtner--Yuzvinsky building set \cite{DCP95,FY04}, the join map \(\prod_{a=1}^{\ell}[\emptyset,F_a]\to[\emptyset,F]\) is an isomorphism of posets; consequently the flats \(F_1,\dots,F_\ell\) are pairwise disjoint with union \(F\), and the restriction \(N|F\) is the direct sum of the restrictions \(N|F_1,\dots,N|F_\ell\). By the direct-sum additivity of the BEST tautological quotient classes \cite[Proposition~5.13]{BEST23}, applied to this direct-sum decomposition with the toric morphisms \(p_a\colon X_F\to X_{F_a}\) induced by the coordinate projections, one has \([Q_{N|F}]=\sum_{a=1}^{\ell}p_a^*[Q_{N|F_a}]\) in \(K_T(X_F)\). Pulling this equality back by \(\pr_1\) and substituting yields the second displayed identity. We do not assert any isomorphism \(X_F\cong X_{F_1}\times\cdots\times X_{F_\ell}\); the dependence on the factors enters only through the toric morphisms \(p_a\).
\end{proof}

\begin{lem}[Orbit-closure restriction identity for BEST Newton power sums]\label{lem:orbit-divisor-product}
Let \(N\) be a loopless matroid on \(E\) of rank \(r=d+1\). Let \(\Gg_{\mathrm{small}}\) be a Feichtner--Yuzvinsky building set in \(L(N)\setminus\{\emptyset\}\) containing \(E\). Let \(F\) be a nonempty proper flat with \(F\notin\Gg_{\mathrm{small}}\), and suppose that \(\Gg_{\mathrm{big}}=\Gg_{\mathrm{small}}\cup\{F\}\) is again a building set. Write
\[
\Fact_{\Gg_{\mathrm{small}}}(F)=\{F_1,\dots,F_\ell\}.
\]
Let \(Z_F\subseteq X_E\) be the permutohedral torus-orbit closure attached to the one-term chain \(\emptyset\subsetneq F\subsetneq E\), with the standard product identification \(Z_F\cong X_F\times X_{E\setminus F}\) and projections \(\pr_1\colon Z_F\to X_F\) and \(\pr_2\colon Z_F\to X_{E\setminus F}\). For each \(a\), let \(p_a\colon X_F\to X_{F_a}\) be the toric morphism induced by the coordinate projection. Let \(Q_N\), \(Q_{N|F_a}\), and \(Q_{N/F}\) be the BEST tautological quotient classes, and let \(P_m(\cdot)\) denote the \(m\)-th Newton power sum of the corresponding total equivariant Chern polynomial. Then, for every \(m\geq 1\), in the equivariant Chow ring of \(Z_F\),
\[
P_m(Q_N|Z_F)=\pr_2^*P_m(Q_{N/F})+\sum_{a=1}^{\ell}\pr_1^*p_a^*P_m(Q_{N|F_a}).
\]
\end{lem}
\begin{proof}
By Lemma~\ref{lem:best-restriction}, in the equivariant K-theory of \(Z_F\),
\[
[Q_N]\big|_{Z_F}=\pr_2^*[Q_{N/F}]+\sum_{a=1}^{\ell}\pr_1^*p_a^*[Q_{N|F_a}].
\]
The total equivariant Chern polynomial \(c_t\) is multiplicative under direct sums of equivariant K-classes and commutes with pullback. Hence, in the equivariant Chow ring of \(Z_F\),
\[
c_t(Q_N|Z_F)=\pr_2^*c_t(Q_{N/F})\cdot\prod_{a=1}^{\ell}\pr_1^*p_a^*c_t(Q_{N|F_a}).
\]
Work in a splitting algebra over the equivariant Chow ring of \(Z_F\) in which each pulled-back total Chern polynomial splits into commuting degree-one Chern roots. Let the Chern roots of \(\pr_2^*c_t(Q_{N/F})\) be \(\alpha_{R,1},\dots,\alpha_{R,q_R}\), and let the Chern roots of \(\pr_1^*p_a^*c_t(Q_{N|F_a})\) be \(\alpha_{a,1},\dots,\alpha_{a,q_a}\). By the displayed product, the Chern roots of \(c_t(Q_N|Z_F)\) are the union of these. Over \(\Qq\), Newton's identities identify the \(m\)-th Newton power sum of a unit total Chern polynomial with the sum of the \(m\)-th powers of its Chern roots, and this identification commutes with the graded pullbacks \(\pr_1^*\), \(\pr_2^*\), and \(p_a^*\). Therefore
\[
P_m(Q_N|Z_F)=\sum_{b=1}^{q_R}\alpha_{R,b}^m+\sum_{a=1}^{\ell}\sum_{b=1}^{q_a}\alpha_{a,b}^m
=\pr_2^*P_m(Q_{N/F})+\sum_{a=1}^{\ell}\pr_1^*p_a^*P_m(Q_{N|F_a})
\]
for every \(m\geq 1\), as claimed.
\end{proof}

\begin{prop}[Descended quotient and tangent class]\label{prop:quotient-tangent}
Let \(N\) be a loopless matroid of rank \(r=d+1\) on a finite nonempty ground set \(E\), with finite lattice of flats, bottom flat \(\emptyset\), and top flat \(E\). Let \(\Hh\subset L(N)\setminus\{\emptyset\}\) be a Feichtner--Yuzvinsky building set containing \(E\), and put \(\Hh^\circ:=\Hh\setminus\{E\}\). Then there exists a unique descended quotient Chern polynomial
\[
C_Q(N,\Hh;u)=\sum_{i=0}^d c_i^Q u^i\in A_{\Qq}(N,\Hh)[u]
\]
whose pullback to the maximal building set is the Berget--Eur--Spink--Tseng quotient Chern polynomial \cite{BEST23}. The polynomial \(C_Q(N,\Hh;u)\) determines a unique rational K-class \(Q_{N,\Hh}\in K_{\Qq}(N,\Hh)\). Define
\[
T_{N,\Hh}:=\sum_{F\in\Hh^\circ}(1-\tau_F)^{-1}-Q_{N,\Hh}.
\]
Then \(Q_{N,\Hh}\) and \(T_{N,\Hh}\) are well-defined classes in \(K_{\Qq}(N,\Hh)\), and
\[
c_z(T_{N,\Hh})=c_z(Q_{N,\Hh})^{-1}\prod_{F\in\Hh^\circ}(1+z x_F),
\]
where \(c_z(Q_{N,\Hh})=\sum_{i=0}^d c_i^Q z^i\). If \(N\) is realized by a complex linear subspace \(L\), and
\[
\theta\colon K_{\Qq}(N,\Hh)\to K_0(W_{L,\Hh})\otimes_{\Zz}\Qq
\]
is the realizable rational K-theory identification matching \((1-\tau_F)^{-1}\) with \([\Oo_{W_{L,\Hh}}(D_F)]\), then
\[
\theta(T_{N,\Hh})=[T_{W_{L,\Hh}}].
\]
No integral \(K_{\Zz}\)-membership assertion is made.
\end{prop}

\begin{proof}
Throughout write \(\Hh_{\max}:=L(N)\setminus\{\emptyset\}\) for the maximal building set. The proof proceeds in six steps: (0) the maximal model and its zero-extended Newton power sums; (1) the one-step image criterion and the injectivity of one-step Chow pullback; (2) the descent of the Newton power sums; (3) the descended quotient Chern polynomial; (4) the reconstruction of \(Q_{N,\Hh}\) through the filtered Chern character; (5) the tangent class and its Chern polynomial; (6) the realizable identification.

\smallskip
\noindent\emph{Step 0 (maximal model).} Let
\[
C_{\mathrm{BEST}}^{\max}(N;u)=\sum_{i=0}^{d}c_i^{\mathrm{BEST}}u^i
\in A_{\Qq}(N,\Hh_{\max})[u]
\]
be the non-equivariant image, on the maximal model \(\Hh_{\max}\), of the Berget--Eur--Spink--Tseng quotient Chern polynomial \(C_{\mathrm{BEST}}(N;u)\) of Definition~\ref{defn:best-quotient}. For \(m\ge 1\) define the \emph{zero-extended} maximal Newton power sum \(P_m^{\max}\in A_{\Qq}(N,\Hh_{\max})^m\) to be the \(m\)-th Newton power sum determined by \(C_{\mathrm{BEST}}^{\max}(N;u)\) when \(1\le m\le d\), and \(P_m^{\max}:=0\) when \(m>d\). Since \(A_{\Qq}(N,\Gg)^i=0\) for \(i>d\) and every Chow pullback is graded and unital, each class named below vanishes in degrees \(>d\); we nevertheless carry all \(m\ge 1\), because the minors arising in one descent step can have strictly smaller top Chow degree, where the same index \(m\) is out of range and the correct representative is \(0\).

\begin{humancomment}
In the realizable maximal case, this construction has the following geometric meaning. If \(N\) is realized by a linear subspace \(L\), then the wonderful compactification \(W_L\) embeds in the permutohedral toric variety \(X_E\). 

The maximal BEST quotient Chern polynomial \(C_{\mathrm{BEST}}^{\max}(N;u)\) is the total Chern polynomial of \(i^*Q_L\), where \(i:W_L\hookrightarrow X_E\). The vector bundle \(Q_L\) on \(X_E\) admits a regular section whose vanishing locus is \(W_L\). Hence, \(i^*Q_L\) is the normal bundle of \(W_L\) in \(X_E\). 

\end{humancomment}
\smallskip
\noindent\emph{Step 1 (one-step Chow pullback: injectivity and image criterion).} Let
\[
\Gg_{\mathrm{big}}=\Gg_{\mathrm{small}}\cup\{F\}
\]
be a one-flat enlargement of top-containing Feichtner--Yuzvinsky building sets, with \(\Fact_{\Gg_{\mathrm{small}}}(F)=\{F_1,\dots,F_\ell\}\), and let
\[
\psi_A\colon A_{\Qq}(N,\Gg_{\mathrm{small}})\longrightarrow A_{\Qq}(N,\Gg_{\mathrm{big}})
\]
be the one-step Chow pullback. Under the Feichtner--Yuzvinsky identification of \(A_{\Qq}(N,\Gg)\) with the Chow ring of the smooth toric variety \(X_\Gg\) of the nested fan of \(\Gg\) \cite{FY04}, the nested fan of \(\Gg_{\mathrm{big}}\) is the stellar subdivision of that of \(\Gg_{\mathrm{small}}\) at the cone spanned by \(\Fact_{\Gg_{\mathrm{small}}}(F)\), with new ray \(F\) \cite{DCP95,FY04}; and \(\psi_A\) is the pullback along the resulting proper birational toric morphism \(\pi\colon X_{\Gg_{\mathrm{big}}}\to X_{\Gg_{\mathrm{small}}}\) of smooth varieties. The projection formula gives \(\pi_*\pi^*=\mathrm{id}\) on rational Chow classes, because \(\pi_*[X_{\Gg_{\mathrm{big}}}]=[X_{\Gg_{\mathrm{small}}}]\); hence \(\psi_A=\pi^*\) is injective. More generally, for any top-containing \(\Gg\subseteq\Hh_{\max}\) the nested fan of \(\Hh_{\max}\) refines that of \(\Gg\), giving a canonical proper birational morphism \(X_{\Hh_{\max}}\to X_\Gg\) of which every one-flat refinement chain from \(\Gg\) to \(\Hh_{\max}\) is a factorization into stellar subdivisions \cite{DCP95,FY04}. Consequently the composite Chow pullback \(\psi_\Gg\colon A_{\Qq}(N,\Gg)\to A_{\Qq}(N,\Hh_{\max})\) equals \((X_{\Hh_{\max}}\to X_\Gg)^*\); it is one and the same map for every chain, and it is injective, being a composition of injective one-step pullbacks.

We use the following \emph{one-step image criterion} (sufficient direction): write \(j\colon E_F\hookrightarrow X_{\Gg_{\mathrm{big}}}\) for the exceptional divisor of \(\pi\) and \(p\colon E_F\to Z_B\) for its projection to the blown-up center \(Z_B\). A class \(y\in A_{\Qq}(N,\Gg_{\mathrm{big}})^m\) lies in \(\operatorname{im}\psi_A\) whenever its exceptional restriction satisfies \(j^*y=p^*\gamma\) for some Chow class \(\gamma\) on \(Z_B\). This is the image description of the toric stellar subdivision \(\pi\): a class on the blowup descends to \(X_{\Gg_{\mathrm{small}}}\) once its restriction to the exceptional divisor is pulled back from the center \cite{DCP95,FY04}.

\smallskip
\noindent\emph{Step 2 (descent of the Newton power sums).} We construct, for every top-containing \(\Gg\) and every \(m\ge1\), a class \(P_m(N,\Gg)\in A_{\Qq}(N,\Gg)^m\), zero for \(m>d\), pulling back to \(P_m^{\max}\) along every one-flat refinement chain. We argue by simultaneous induction on \(|E|\), and inside a fixed \(N\) by descending along a one-flat refinement chain
\[
\Hh=\Gg_0\subset\Gg_1\subset\cdots\subset\Gg_s=\Hh_{\max}.
\]
At the maximal end set \(P_m(N,\Gg_s):=P_m^{\max}\). Suppose \(P_m(N,\Gg_j)\) is constructed with the maximal-pullback property and \(\Gg_j=\Gg_{j-1}\cup\{F\}\); assume \(1\le m\le d\), since otherwise both sides vanish. We compute the exceptional restriction \(j^*P_m(N,\Gg_j)\) by pulling it to the maximal model and using ordinary Chow functoriality, not divisor base change. Let \(\varpi\colon X_{\Hh_{\max}}\to X_{\Gg_j}\) be the chain morphism of Step~1, let \(\Delta:=V_{\Sigma(\Hh_{\max})}(F)\) be the strict transform of \(E_F\) in the maximal model, that is, the divisor star of the ray \(F\) in the maximal nested fan, with inclusion \(j_{\max}\colon\Delta\hookrightarrow X_{\Hh_{\max}}\), and let \(g\colon\Delta\to E_F\) be the toric morphism induced by the refinement of star fans \(\operatorname{Star}_{\Sigma(\Hh_{\max})}(F)\to\operatorname{Star}_{\Sigma(\Gg_j)}(F)\) in the quotient lattice \(N_{\mathbb Z}/\mathbb Z F\) \cite{DCP95,FY04}. For the permutohedral maximal model this star is the product fan, so \(\Delta\) is the orbit divisor \(Z_F\cong X_F\times X_{E\setminus F}\subseteq X_E\) attached to \(\emptyset\subsetneq F\subsetneq E\). These maps form the strict-transform square \(\varpi\circ j_{\max}=j\circ g\colon\Delta\to X_{\Gg_j}\), with vertical maps \(g\colon\Delta\to E_F\) and \(\varpi\colon X_{\Hh_{\max}}\to X_{\Gg_j}\) and horizontal inclusions \(j_{\max}\colon\Delta\hookrightarrow X_{\Hh_{\max}}\) and \(j\colon E_F\hookrightarrow X_{\Gg_j}\). It commutes but is in general \emph{not} Cartesian: the scheme-theoretic preimage \(\varpi^{-1}(E_F)\) acquires extra exceptional components from later stellar subdivisions whose centers contain the ray \(F\), so \(\Delta\) is only the strict transform of \(E_F\), not the total transform. Since all four varieties are smooth, ordinary contravariant Chow functoriality applies to the commuting square and gives \(g^*j^*=j_{\max}^*\varpi^*\) on \(A_{\Qq}(X_{\Gg_j})\). Applying this to \(P_m(N,\Gg_j)\) and using \(\varpi^*P_m(N,\Gg_j)=P_m^{\max}\), together with the descent property \(j_{\max}^*P_m^{\max}=P_m(Q_N|Z_F)\) restricting the \(m\)-th BEST Newton power sum to the orbit divisor \(\Delta=Z_F\), yields
\[
g^*\big(j^*P_m(N,\Gg_j)\big)=P_m(Q_N|Z_F)\qquad\text{in }A_{\Qq}(\Delta).
\]

The restriction minors \(N|F_a\) and the contraction minor \(N/F\) have ground sets strictly smaller than \(E\), and carry the induced top-containing building sets \(\Gg_{\mathrm{small}}|F_a\) and \(\Gg_{\mathrm{small}}/F\). By the induction on \(|E|\), their descended Newton power sums exist as ordinary Chow classes: let \(P_{a,m}\in A_{\Qq}(N|F_a,\Gg_{\mathrm{small}}|F_a)^m\) and \(P_{R,m}\in A_{\Qq}(N/F,\Gg_{\mathrm{small}}/F)^m\) be the classes pulling back to \(P_m(Q_{N|F_a})\) and \(P_m(Q_{N/F})\) respectively along the corresponding maximal refinement chains (with the zero convention when \(m\) is out of range). Let
\[
\Theta\colon A_{\Qq}(Z_B)\xrightarrow{\ \sim\ }\Big(\bigotimes_{a=1}^{\ell}A_{\Qq}(N|F_a,\Gg_{\mathrm{small}}|F_a)\Big)\otimes_{\Qq}A_{\Qq}(N/F,\Gg_{\mathrm{small}}/F)
\]
be the center product Chow isomorphism, and define \(\gamma_m\in A_{\Qq}(Z_B)^m\) by
\[
\Theta(\gamma_m):=P_{R,m}+\sum_{a=1}^{\ell}P_{a,m},
\]
each summand embedded in its own tensor factor with the unit on the others. This is an equality of ordinary Chow classes; no piecewise-polynomial representative and no equivariant lift is invoked. The toric morphism \(p\circ g\colon\Delta\to Z_B\) carries \(\Delta=Z_F\) to the center, and under the product identifications its pullback of \(\Theta(\gamma_m)\) is, by Lemma~\ref{lem:orbit-divisor-product},
\[
P_m(Q_N|Z_F)=\pr_2^*P_m(Q_{N/F})+\sum_{a=1}^{\ell}\pr_1^*p_a^*P_m(Q_{N|F_a})=g^*p^*\gamma_m.
\]
Comparing with the exceptional-restriction identity of the previous paragraph gives \(g^*\big(j^*P_m(N,\Gg_j)\big)=g^*p^*\gamma_m\) in \(A_{\Qq}(\Delta)\). The star-fan refinement morphism \(g\colon\Delta\to E_F\) is a proper birational toric morphism of smooth varieties, being the restriction to divisor stars of the stellar-subdivision chain \(\varpi\); the projection-formula argument of Step~1 applied to \(g\) gives \(g_*g^*=\mathrm{id}\) on rational Chow classes, so \(g^*\colon A_{\Qq}(E_F)\to A_{\Qq}(\Delta)\) is injective. We conclude
\[
j^*P_m(N,\Gg_j)=p^*\gamma_m\qquad\text{in }A_{\Qq}(E_F).
\]
The image criterion of Step~1 now gives \(P_m(N,\Gg_j)\in\operatorname{im}\psi_j\), and injectivity of \(\psi_j\) yields a unique \(P_m(N,\Gg_{j-1})\) with \(\psi_j(P_m(N,\Gg_{j-1}))=P_m(N,\Gg_j)\). Since the composite from \(\Gg_{j-1}\) to \(\Hh_{\max}\) factors through \(\Gg_j\), the new class again has the maximal-pullback property. Descending to \(j=0\) produces \(P_m(N,\Hh)\).

By Step~1 the composite \(\psi_\Hh\colon A_{\Qq}(N,\Hh)\to A_{\Qq}(N,\Hh_{\max})\) is one chain-independent injective map, and a class has the maximal-pullback property exactly when \(\psi_\Hh\) sends it to \(P_m^{\max}\). If two classes of \(A_{\Qq}(N,\Hh)^m\) both have this property, their difference lies in \(\ker\psi_\Hh=0\), so they agree. Hence \(P_m(N,\Hh)\) is the unique class with \(\psi_\Hh(P_m(N,\Hh))=P_m^{\max}\), and is automatically independent of the chain. For a one-flat enlargement, \(\psi_{\Gg_{\mathrm{big}}}\circ\psi_A=\psi_{\Gg_{\mathrm{small}}}\), so both \(\psi_A(P_m(N,\Gg_{\mathrm{small}}))\) and \(P_m(N,\Gg_{\mathrm{big}})\) are sent by the injective map \(\psi_{\Gg_{\mathrm{big}}}\) to \(P_m^{\max}\), whence
\[
\psi_A(P_m(N,\Gg_{\mathrm{small}}))=P_m(N,\Gg_{\mathrm{big}}).
\]
\begin{humancomment}
The idea of Step 2 can be read by the following commutative diagram.
\begin{equation}
\begin{tikzcd}
\Delta \arrow[r, hookrightarrow, "j_{\max}"] \arrow[d, "g"] &
X_{\Hh_{\max}} \arrow[d,  "\varpi"] \\
E_F \arrow[r, hookrightarrow, "j"]\arrow[d, "p"] & X_{\Gg_{\mathrm{big}}}\arrow[d] \\
Z_B \arrow[r, hookrightarrow] & X_{\Gg_{\mathrm{small}}}
\end{tikzcd}\end{equation}

Recall that \(C_{\mathrm{BEST}}^{\max}(N;u)\) is the Chern class of the pullback of the BEST class. Let \(P_{\mathrm{BEST}}^{\max}(N;u)\) denote the Newton power sum (the normalized Chern character) of the pullback of the BEST class. The goal is to show that the class
\(P(N,\Gg_{\mathrm{big}})\) descends across the one-flat blowdown
\(X_{\Gg_{\mathrm{big}}}\to X_{\Gg_{\mathrm{small}}}\). 

At this stage of the descending induction, put
\(y=P(N,\Gg_{\mathrm{big}})\), and
\[
g^*j^*y=j_{\max}^*\varpi^* y=j_{\max}^*P_{\mathrm{BEST}}^{\max}.
\]

By Lemma~\ref{lem:best-restriction}, this last class is the sum of the
corresponding Newton power sums for the minors \(N|F_a\) and \(N/F\). \[
j_{\max}^\ast(P_{\mathrm{BEST}}^{\max}(N;u))=\pr_2^*(P_{\mathrm{BEST}}^{\max}(N/F;u))+\sum_{a=1}^{\ell}\pr_1^*p_a^*(P_{\mathrm{BEST}}^{\max}(N|F_a;u)).
\]

By induction on the smaller ground sets, these minor classes descend, and assemble to a class \(\gamma\in A_{\mathbb Q}(Z_B)\)
with
\[
g^*p^*\gamma=j_{\max}^*P_{\mathrm{BEST}}^{\max}.
\]
Since \(g^*\) is injective, \(j^*y=p^*\gamma_m\). The one-step image criterion
then implies that \(y\) descends to \(A_{\mathbb Q}(N,\Gg_{\mathrm{small}})\).
\end{humancomment}
\smallskip
\noindent\emph{Step 3 (descended quotient Chern polynomial).} Define \(c_0^Q=1\) and, for \(1\le m\le d\), the coefficient \(c_m^Q\in A_{\Qq}(N,\Hh)^m\) by Newton's recursion
\[
P_m-c_1^QP_{m-1}+c_2^QP_{m-2}-\cdots+(-1)^{m-1}c_{m-1}^QP_1+(-1)^m m\,c_m^Q=0,
\qquad P_j:=P_j(N,\Hh),
\]
which determines \(c_m^Q\) uniquely because \(m\) is invertible in \(\Qq\). Put \(C_Q(N,\Hh;u):=\sum_{i=0}^{d}c_i^Q u^i\). Applying the chain pullback to the recursion and using that \(P_m(N,\Hh)\) pulls back to \(P_m^{\max}\) shows that \(C_Q(N,\Hh;u)\) pulls back to \(C_{\mathrm{BEST}}^{\max}(N;u)\). Any polynomial with this same maximal pullback has each coefficient differing from \(c_i^Q\) by a class killed by an injective chain pullback, hence equals \(C_Q(N,\Hh;u)\); this is the asserted existence, uniqueness and chain-independence.

\smallskip
\noindent\emph{Step 4 (reconstruction of \(Q_{N,\Hh}\) via the filtered Chern character).} Put \(q_Q:=|\Hh^\circ|-d\), let \(p_m^Q\) be the Newton power sums determined by \(C_Q(N,\Hh;u)\) (so \(p_m^Q=P_m(N,\Hh)\)), and set
\[
\ch_Q^A:=q_Q\cdot 1+\sum_{m=1}^{d}\frac{p_m^Q}{m!}\in A_{\Qq}(N,\Hh).
\]
Filter \(K_{\Qq}(N,\Hh)\) by the powers \(I_\tau^p\) of the ideal \(I_\tau\) generated by the \(\tau_F\), and \(A_{\Qq}(N,\Hh)\) by \(F_A^p=\bigoplus_{i\ge p}A_{\Qq}(N,\Hh)^i\). The assignment \(x_F\mapsto\operatorname{in}(\tau_F)\in I_\tau/I_\tau^2\) carries the defining relations of \(A_{\Qq}(N,\Hh)\) to the initial forms of the defining relations of \(K_{\Qq}(N,\Hh)\): a nonnested monomial \(\prod_{F\in S}x_F\) is the initial form of \(\prod_{F\in S}\tau_F\), and the atom linear form \(\sum_{F\ni a}x_F\) is the degree-one part of the multiplicative atom relation \(1-\prod_{F\ni a}(1-\tau_F)\). By the standard-monomial basis of the combinatorial K-ring \cite{LLPP24}, this assignment is an isomorphism of graded \(\Qq\)-algebras
\[
\operatorname{gr}_{I_\tau}K_{\Qq}(N,\Hh)\cong A_{\Qq}(N,\Hh),
\]
and \(I_\tau\) is nilpotent with \(I_\tau^{d+1}=0\); the \(I_\tau\)-adic filtration is therefore finite and separated, and \(\dim_{\Qq}K_{\Qq}(N,\Hh)=\sum_{i=0}^{d}\dim_{\Qq}A_{\Qq}(N,\Hh)^i\). Since \(\ch(\tau_F)=1-\exp(-x_F)=x_F+(\text{Chow degree}\ge2)\), the Chern character carries \(I_\tau^p\) into \(F_A^p\), so it is filtered, and its associated graded map is the inverse of the displayed isomorphism. A filtered homomorphism inducing an isomorphism on the associated graded of finite filtrations is itself an isomorphism. Hence \(\ch\) is a \(\Qq\)-algebra isomorphism, and
\[
Q_{N,\Hh}:=\ch^{-1}(\ch_Q^A)\in K_{\Qq}(N,\Hh)
\]
is the unique class with \(\ch(Q_{N,\Hh})=\ch_Q^A\).
\begin{humancomment} Step 3 and Step 4 reconstruct the (rational) \(K\)-class by its Chern character. This is possible because the Chern character map \(\ch:K(X)_\mathbb{Q}\xrightarrow{\sim}A(X)_\mathbb{Q} \) induces an isomorphism when the variety \(X\) is smooth, and the toric variety corresponding to the fan gives the isomorphism. 
   
\end{humancomment}

\smallskip
\noindent\emph{Step 5 (tangent class and its Chern polynomial).} Each \(\tau_F\) lies in the nilpotent ideal \(I_\tau\), so \(1-\tau_F\) is a unit with \((1-\tau_F)^{-1}=\sum_{k\ge0}\tau_F^k\), a finite sum. Hence
\[
T_{N,\Hh}:=\sum_{F\in\Hh^\circ}(1-\tau_F)^{-1}-Q_{N,\Hh}
\]
is well-defined in \(K_{\Qq}(N,\Hh)\). Recall that \(K_{\Qq}(N,\Hh)\) is a special \(\lambda\)-ring \cite{BEST23,LLPP24}, so each class \(a\) has a total Chern class \(c_z(a)=\sum_i c_i(a)z^i\) with \(c_i(a)\in A_{\Qq}(N,\Hh)^i\), recovered from the Chern-character power sums \(p_m(a):=m!\,\ch_m(a)\) by Newton's identities; this total Chern class is multiplicative, \(c_z(a+b)=c_z(a)c_z(b)\) and \(c_z(-a)=c_z(a)^{-1}\) by the splitting principle, and a line element \(l\) with \(\ch(l)=\exp(u)\) has \(c_z(l)=1+zu\). The boundary class \((1-\tau_F)^{-1}\) is a unit with \(\ch((1-\tau_F)^{-1})=\exp(x_F)\), hence a line element with \(c_z((1-\tau_F)^{-1})=1+z x_F\); and \(c_z(Q_{N,\Hh})=\sum_{i=0}^{d}c_i^Q z^i\), since the \(c_i^Q\) are the Newton-identity Chern classes of \(\ch_Q^A\). Therefore
\[
c_z(T_{N,\Hh})=c_z(Q_{N,\Hh})^{-1}\prod_{F\in\Hh^\circ}(1+z x_F),
\]
the \(z\)-graded inverse existing because \(c_0^Q=1\).

\smallskip
\noindent\emph{Step 6 (realizable identification).} Suppose \(N\) is realized by a complex linear subspace \(L\subseteq\mathbb C^E\) not contained in any coordinate hyperplane, with \(W:=W_{L,\Hh}\) the smooth projective De Concini--Procesi wonderful model and reduced simple normal crossings boundary \(D=\sum_{F\in\Hh^\circ}D_F\). Write \(W_{\max}:=W_{L,\Hh_{\max}}\) for the maximal model, with reduced boundary \(D_{\max}\). Consider the shifted geometric quotient class
\[
Q_{\mathrm{sum}}:=(|E|-1)[\Oo_W]-[T_W(-\log D)]+(|\Hh^\circ|-|E|+1)[\Oo_W]=|\Hh^\circ|\,[\Oo_W]-[T_W(-\log D)]
\]
in \(K_0(W)\otimes\Qq\), built from the logarithmic tangent bundle of \((W,D)\). We first prove that, under the realizable Chow identification \(x_F\leftrightarrow[D_F]\), the total Chern polynomial of \(Q_{\mathrm{sum}}\) equals \(C_Q(N,\Hh;u)\). For the maximal building set this is the Berget--Eur--Spink--Tseng identity \cite[Theorem~8.8]{BEST23}; the passage to an arbitrary building set is effected by the same one-flat refinement that defines \(C_Q\) in Steps~1--3. We carry it out in detail.

\emph{(6a) Maximal anchor.} Let \(Q_L\) denote the Berget--Eur--Spink--Tseng tautological quotient bundle on the permutohedral variety, restricted to \(W_{\max}\). Restricting the tautological exact sequence of Definition~\ref{defn:best-quotient} to \(W_{\max}\) and applying \cite[Theorem~8.8]{BEST23} gives
\[
[Q_L|_{W_{\max}}]=(|E|-1)[\Oo_{W_{\max}}]-[T_{W_{\max}}(-\log D_{\max})]
\]
in \(K_0(W_{\max})\), whose total Chern polynomial, under the permutohedral Feichtner--Yuzvinsky identification, is \(C_{\mathrm{BEST}}^{\max}(N;u)\) of Step~0.

\emph{(6b) Refinement geometry.} Order the flats of \(\Hh_{\max}\setminus\Hh\) so that adjoining them to \(\Hh\) one at a time yields a Feichtner--Yuzvinsky building set at each stage; this produces a chain \(\Hh=\Gg_0\subset\Gg_1\subset\cdots\subset\Gg_s=\Hh_{\max}\) and a composite morphism
\[
\pi\colon W_{\max}=W_{L,\Gg_s}\xrightarrow{\,b_s\,}\cdots\xrightarrow{\,b_1\,}W_{L,\Gg_0}=W,
\]
in which \(b_j\colon W_{L,\Gg_j}\to W_{L,\Gg_{j-1}}\) is the blow-up of \(W_{L,\Gg_{j-1}}\) along the smooth center cut out by the boundary divisors indexed by the maximal elements \(\Fact_{\Gg_{j-1}}(F)=\{F_a\}\) of \(\Gg_{j-1}\) below the newly added flat \(F\); this center is the transverse intersection \(\bigcap_a D_{F_a}\), and the reduced boundary of \(W_{L,\Gg_j}\) is the total transform of that of \(W_{L,\Gg_{j-1}}\) together with the new exceptional divisor \(D_F\) \cite{DCP95}, \cite[Proposition~5.2]{EFMPV25}. Thus each \(b_j\) is the blow-up of a smooth stratum of a simple normal crossings divisor, the boundary remains simple normal crossings, and \(\pi\) is their composite.

\emph{(6c) Blow-up invariance of the shifted log-normal class.} Let \(b\colon W'\to W''\) be the blow-up of a smooth transverse intersection of components of a simple normal crossings divisor \(D''\) on \(W''\), with reduced total transform \(D'\) (the strict transforms together with the exceptional divisor) again simple normal crossings. Such a blow-up is log-\'etale for the divisorial log structures \((W'',D'')\) and \((W',D')\); equivalently the logarithmic cotangent sheaf pulls back, \(b^*\Omega^1_{W''}(\log D'')\cong\Omega^1_{W'}(\log D')\), and dually \(b^*[T_{W''}(-\log D'')]=[T_{W'}(-\log D')]\) in \(K_0(W')\). Since also \(b^*[\Oo_{W''}]=[\Oo_{W'}]\), the class \((|E|-1)[\Oo]-[T(-\log D)]\) is compatible with \(b^*\). Applying this to each \(b_j\) of (6b) and composing, and using (6a),
\[
\pi^*\big((|E|-1)[\Oo_W]-[T_W(-\log D)]\big)=(|E|-1)[\Oo_{W_{\max}}]-[T_{W_{\max}}(-\log D_{\max})]=[Q_L|_{W_{\max}}].
\]

\emph{(6d) Descent of the Chern polynomial.} For any virtual class \(V\) and integer \(a\) one has \(c_z(V+a[\Oo])=c_z(V)\), since \(c_z(\Oo)=1\) and \(c_z\) is multiplicative; hence \(\pi^*Q_{\mathrm{sum}}\) and \([Q_L|_{W_{\max}}]\) have the same total Chern polynomial, namely \(C_{\mathrm{BEST}}^{\max}(N;u)\) by (6a). Under the Feichtner--Yuzvinsky presentations the geometric pullback \(\pi^*\) on rational Chow rings coincides with the iterated stellar-subdivision Chow pullback \(\psi_\Hh\colon A_{\Qq}(N,\Hh)\to A_{\Qq}(N,\Hh_{\max})\) of Step~1 \cite{FY04}, \cite[Proposition~5.2]{EFMPV25}, and Chern classes commute with pullback; so \(\psi_\Hh\) sends the total Chern polynomial of \(Q_{\mathrm{sum}}\) to \(C_{\mathrm{BEST}}^{\max}(N;u)\). By Step~3, \(C_Q(N,\Hh;u)\) is the unique element of \(A_{\Qq}(N,\Hh)[u]\) whose \(\psi_\Hh\)-image is \(C_{\mathrm{BEST}}^{\max}(N;u)\), and \(\psi_\Hh\) is injective; therefore, under \(x_F\leftrightarrow[D_F]\), the total Chern polynomial of \(Q_{\mathrm{sum}}\) equals \(C_Q(N,\Hh;u)\). For \(\Hh=\Hh_{\max}\) this is exactly \cite[Theorem~8.8]{BEST23}.

\emph{(6e) Identification of the classes.} Since \(W\) is smooth and projective, the topological Chern character \(\ch_W\colon K_0(W)\otimes\Qq\xrightarrow{\ \sim\ }\mathrm{CH}^*(W)\otimes\Qq\) is a ring isomorphism, and \(\theta\) intertwines \(\ch\) with \(\ch_W\) under \(x_F\leftrightarrow[D_F]\). As \(\ch(Q_{N,\Hh})=\ch_Q^A\) is the Chern character determined by \(C_Q(N,\Hh;u)\), and by (6d) \(\ch_W(Q_{\mathrm{sum}})\) is determined by the same polynomial, the classes \(\theta(Q_{N,\Hh})\) and \(Q_{\mathrm{sum}}\) have equal Chern character; injectivity of \(\ch_W\) gives \(\theta(Q_{N,\Hh})=Q_{\mathrm{sum}}\). Finally,
\[
\theta(T_{N,\Hh})=\sum_{F\in\Hh^\circ}[\Oo_W(D_F)]-Q_{\mathrm{sum}}
=[T_W(-\log D)]+\sum_{F\in\Hh^\circ}\big([\Oo_W(D_F)]-[\Oo_W]\big).
\]
For each \(F\), the structure sequence \(0\to\Oo_W\to\Oo_W(D_F)\to\Oo_{D_F}(D_F)\to0\) gives \([\Oo_W(D_F)]-[\Oo_W]=[\Oo_{D_F}(D_F)]\), while the residue sequence \(0\to T_W(-\log D)\to T_W\to\bigoplus_{F\in\Hh^\circ}\Oo_{D_F}(D_F)\to0\) of the simple normal crossings divisor \(D\) gives \([T_W]=[T_W(-\log D)]+\sum_{F\in\Hh^\circ}[\Oo_{D_F}(D_F)]\). Substituting yields
\[
\theta(T_{N,\Hh})=[T_{W_{L,\Hh}}]
\]
in \(K_0(W_{L,\Hh})\otimes_{\Zz}\Qq\). No integral \(K_{\Zz}\)-membership is asserted.
\end{proof}
\begin{humancomment}
We have the pullback map between Chow rings of toric models \[
\psi_A\colon A_{\Qq}(N,\Gg_{\mathrm{small}})\longrightarrow A_{\Qq}(N,\Gg_{\mathrm{big}}),
\] Proposition~\ref{prop:generator-normalization}  implicitly implies that when $M$ is realizable, $\psi_A$ is the same as the pullback map between Chow rings of wonderful models \[
\pi\colon W_{L,\Gg_{\mathrm{big}}}\longrightarrow W_{L,\Gg_{\mathrm{small}}}.
\] See \cite[Theorem 2.1, Remark 3.10]{Cheng26}.
\end{humancomment}
\begin{rem}\label{rem:induced-building-sets}
In the one-flat recursion, the restriction and contraction factors carry the induced building sets. In general, the restriction-side building set may omit the top flat of the restriction, and an induced contraction building set of a minimal building set need not be the minimal building set of the contraction. Thus no proof below replaces these induced building sets by top-containing or minimal building sets without a comparison theorem.
\end{rem}

\section{The integral quotient and tangent class}\label{sec:quotient}

In this section, we construct the integral representative used in Theorem~\ref{thm:main-integral}. The main point is that the quotient class is not reconstructed from rational Chern data by Newton identities.

Let $\Gg_{\max}:=L(M)\setminus\{\emptyset\}$. In the maximal case, $K_{\Zz}(M,\Gg_{\max})$ is the Larson-Li-Payne-Proudfoot matroid K-ring. The Berget-Eur-Spink-Tseng tautological quotient class on the permutohedral model restricts integrally to this ring \cite{BEST23,LLPP24}. If $r=\rk(M)$, $n=|E|$, $d=r-1$, and $f=|\Gg_{\max}\setminus\{E\}|$, the maximal integral quotient representative is
\[
Q_{\mathrm{int}}^{\max}(M):=\iota_M^*([Q_M])+(f-n+1)\cdot 1.
\]
The added trivial summand changes the rank and does not change the positive-degree Chern classes. Its rank becomes
\[
(n-r)+(f-n+1)=f-r+1=f-d.
\]
Thus it matches the quotient rank used by the rational class $Q_{M,\Gg_{\max}}$ of Proposition~\ref{prop:quotient-tangent}.

\begin{prop}[Maximal integral quotient]\label{prop:maximal-quotient}
For the maximal building set $\Gg_{\max}$, the rational class $Q_{M,\Gg_{\max}}$ of Proposition~\ref{prop:quotient-tangent} has the integral representative $Q_{\mathrm{int}}^{\max}(M)$ in $K_{\Zz}(M,\Gg_{\max})$.
\end{prop}

\begin{proof}
The Berget-Eur-Spink-Tseng quotient K-class $[Q_M]$ is integral on the permutohedral variety \cite{BEST23}. Its ordinary non-equivariant restriction to the Larson-Li-Payne-Proudfoot matroid K-ring is an element of $K_{\Zz}(M,\Gg_{\max})$ \cite{LLPP24}. Adding $(f-n+1)\cdot 1$ is an integral operation. After rationalization, this class has rank $f-d$ and the quotient Chern polynomial $C_Q(M,\Gg_{\max};u)=C_{\mathrm{BEST}}^{\max}(M;u)$ used in Step~0 of Proposition~\ref{prop:quotient-tangent} to define $Q_{M,\Gg_{\max}}$. The rational Chern character is determined by the rank and Chern classes, so the rationalization of $Q_{\mathrm{int}}^{\max}(M)$ is $Q_{M,\Gg_{\max}}$.
\end{proof}

For a general building set $\Gg$, the integral quotient is obtained by descent from $\Gg_{\max}$. This descent uses one-flat refinements of building sets. If $\Gg_{\mathrm{big}}=\Gg_{\mathrm{small}}\cup\{F_{\mathrm{new}}\}$ and
\[
B:=\Factor(F_{\mathrm{new}},\Gg_{\mathrm{small}})
\]
is the set of maximal elements of $\Gg_{\mathrm{small}}$ contained in $F_{\mathrm{new}}$, then the integral one-step map is
\[
\phi_K(\tau_H)=
\begin{cases}
\tau_H, & H\notin B,\\
\tau_H+\tau_{F_{\mathrm{new}}}-\tau_H\tau_{F_{\mathrm{new}}}, & H\in B.
\end{cases}
\]
Since $F_{\mathrm{new}}$ is a proper flat, the top flat $E$ is not in $B$, and therefore $\phi_K(\tau_E)=\tau_E$.

\begin{prop}[Integral quotient descent]\label{prop:quotient-descent}
For every top-containing building set $\Gg$, the rational class $Q_{M,\Gg}$ of Proposition~\ref{prop:quotient-tangent} has an integral representative $Q_{\Gg}^{\Zz}\in K_{\Zz}(M,\Gg)$. It may be chosen in the form
\[
Q_{\Gg}^{\Zz}=|\Gg^\circ|-(|E|-1)+q_{\mathrm{BEST}}^{\Zz}(M,\Gg),
\]
where $q_{\mathrm{BEST}}^{\Zz}(M,\Gg)$ is the descended integral Berget-Eur-Spink-Tseng quotient class.
\end{prop}

\begin{proof}
We descend the uncorrected Berget-Eur-Spink-Tseng class, not the rank-corrected maximal quotient. Put
\[
q_{\mathrm{BEST}}^{\max}:=\iota_M^*([Q_M])\in K_{\Zz}(M,\Gg_{\max}).
\]
Choose a one-flat chain from $\Gg$ to $\Gg_{\max}$. The rational Berget-Eur-Spink-Tseng quotient descent of Proposition~\ref{prop:quotient-tangent} (the chain-compatibility $\psi_A(C_Q(M,\Gg_{\mathrm{small}};u))=C_Q(M,\Gg_{\mathrm{big}};u)$ of Step~2--3) says that, at each reverse one-step enlargement, the rationalization of the refined class lies in the rational image of the one-step map. Since Proposition~\ref{prop:one-step-saturation} shows that this image is saturated, the class descends integrally. Repeating along the chain gives $q_{\mathrm{BEST}}^{\Zz}(M,\Gg)\in K_{\Zz}(M,\Gg)$. We then set
\[
Q_{\Gg}^{\Zz}=q_{\mathrm{BEST}}^{\Zz}(M,\Gg)+\bigl(|\Gg^\circ|-(|E|-1)\bigr)\cdot 1,
\]
with the rank correction $\rho_{\Gg}=|\Gg^\circ|-(|E|-1)$ for the building set $\Gg$. The equality $\rho_K(Q_{\Gg}^{\Zz})=Q_{M,\Gg}$ after rationalization follows from the rational quotient descent. The descended class is independent of the chosen one-flat chain: $K_{\Zz}(M,\Gg)$ is torsion-free and $Q_{M,\Gg}$ is a fixed rational class, so its integral preimage under the injective rationalization is unique.
\end{proof}

The next proposition is the linkage between the integral construction and the rational tangent-class theorem of Sections~\ref{sec:tangent-class}--\ref{sec:chern-alpha}.

\begin{prop}[Linkage with the rational tangent class]\label{prop:linkage}
Let $Q_{\Gg}^{\Zz}\in K_{\Zz}(M,\Gg)$ be the integral quotient representative of Proposition~\ref{prop:quotient-descent}, and put $T_{M,\Gg}^{\Zz}:=\sum_{F\in\Gg^\circ}(1-\tau_F)^{-1}-Q_{\Gg}^{\Zz}$. Then
\[
\rho_K(Q_{\Gg}^{\Zz})=Q_{M,\Gg},\qquad
\rho_K(T_{M,\Gg}^{\Zz})=T_{M,\Gg},
\]
where $Q_{M,\Gg}$ and $T_{M,\Gg}$ are the rational quotient and tangent classes of Proposition~\ref{prop:quotient-tangent}.
\end{prop}

\begin{proof}
By Proposition~\ref{prop:quotient-descent}, $\rho_K(Q_{\Gg}^{\Zz})$ is the rational K-class whose rank is $\rho_{\Gg}+(\rk Q_{\mathrm{BEST}})=(|\Gg^\circ|-(|E|-1))+(|E|-1-d)=|\Gg^\circ|-d=q_Q$ and whose total Chern polynomial in positive degrees is the descended Berget-Eur-Spink-Tseng quotient Chern polynomial; the trivial-rank correction $\rho_{\Gg}\cdot 1$ changes only the rank. By Proposition~\ref{prop:quotient-tangent}, $Q_{M,\Gg}$ is the unique rational K-class with rank $q_Q=|\Gg^\circ|-d$ and total Chern polynomial $C_Q(M,\Gg;u)$, the descended Berget-Eur-Spink-Tseng quotient Chern polynomial. A rational K-class is determined by its Chern character, hence by its rank together with its positive-degree Chern classes through Newton's identities. As both classes have the same rank and the same Chern polynomial, $\rho_K(Q_{\Gg}^{\Zz})=Q_{M,\Gg}$. Since rationalization is a ring homomorphism fixing each $(1-\tau_F)^{-1}$, it sends $T_{M,\Gg}^{\Zz}=\sum_{F\in\Gg^\circ}(1-\tau_F)^{-1}-Q_{\Gg}^{\Zz}$ to $\sum_{F\in\Gg^\circ}(1-\tau_F)^{-1}-Q_{M,\Gg}=T_{M,\Gg}$.
\end{proof}

\begin{proof}[Proof of Theorem~\ref{thm:main-integral}, non-realizable clauses]
By Proposition~\ref{prop:quotient-descent}, there exists $Q_{\Gg}^{\Zz}\in K_{\Zz}(M,\Gg)$ with rationalization $Q_{M,\Gg}$ (Proposition~\ref{prop:linkage}). Since every $(1-\tau_F)^{-1}$ is an integral inverse boundary line class in $K_{\Zz}(M,\Gg)$, the class
\[
T_{M,\Gg}^{\Zz}=\sum_{F\in \Gg^\circ}(1-\tau_F)^{-1}-Q_{\Gg}^{\Zz}
\]
lies in $K_{\Zz}(M,\Gg)$. Its rationalization is the rational tangent class $T_{M,\Gg}$ by Proposition~\ref{prop:linkage}. Proposition~\ref{prop:basis} gives integer coordinates in the standard $\tau$-monomial basis.

The equality $P_{\mathrm{int}}^K(M,\Gg;z)=\Hilb(M,\Gg;z)$ and the Chern-alpha inequalities are proved in Section~\ref{sec:hilbert} below, where they are deduced from the in-paper Theorems~\ref{thm:hilbert-identity} and~\ref{thm:chern-alpha} via Proposition~\ref{prop:linkage}.
\end{proof}

\section{Saturated descent}\label{sec:saturation}

The goal of this section is to prove the integral one-step descent statement used in Proposition~\ref{prop:quotient-descent}. The proof uses toric geometry only for auxiliary centers.

\begin{humancomment}
The class pull backs to an integral class in the maximal building set, so it suffices to show that the pullback map has torsion-free cokernel. This can be proved by writing down the map explicitly.
\end{humancomment}

\begin{setup}\label{setup:one-step}
Let $\Gg_{\mathrm{big}}=\Gg_{\mathrm{small}}\cup\{F_{\mathrm{new}}\}$ be a one-flat enlargement of top-containing Feichtner-Yuzvinsky building sets. Put
\[
B:=\Factor(F_{\mathrm{new}},\Gg_{\mathrm{small}})=\{F_1,\ldots,F_\ell\}.
\]
Let $K_{\mathrm{small}}:=K_{\Zz}(M,\Gg_{\mathrm{small}})$ and $K_{\mathrm{big}}:=K_{\Zz}(M,\Gg_{\mathrm{big}})$, and let $\phi_K\colon K_{\mathrm{small}}\to K_{\mathrm{big}}$ be the one-step map.
\end{setup}

\begin{lem}[Associated graded]\label{lem:assoc-graded}
For every top-containing building set $\Gg$ the $\tau$-adic filtration of $K_{\Zz}(M,\Gg)$ by powers of the ideal $I_\tau=(\tau_F:F\in\Gg)$ is finite and separated, and there is a canonical isomorphism
\[
\operatorname{gr}_\tau K_{\Zz}(M,\Gg)\cong A_{\Zz}(M,\Gg),\qquad \operatorname{in}(\tau_F)=x_F.
\]
\end{lem}

\begin{proof}
The nonnested relation $\prod_{F\in S}\tau_F=0$ is homogeneous and has initial form $\prod_{F\in S}x_F=0$. The atom-character relation $1-\prod_{a\leq F}(1-\tau_F)=0$ has lowest-degree term $\sum_{a\leq F}\tau_F$, with initial form $\sum_{a\leq F}x_F=0$, the Feichtner-Yuzvinsky atom-linear relation. These are exactly the defining relations of $A_{\Zz}(M,\Gg)$ \cite{FY04}, and the K-relations are inhomogeneous deformations whose lowest-order terms recover them \cite{LLPP24}. The filtration is finite because $A_{\Zz}(M,\Gg)$ is concentrated in degrees $0,\ldots,d$.
\end{proof}

\begin{lem}[Integral one-step Chow comparison]\label{lem:chow-comparison}
In the one-step situation of Set-up~\ref{setup:one-step}, the Feichtner-Yuzvinsky homomorphism
\[
\psi_{\Zz}\colon A_{\Zz}(M,\Gg_{\mathrm{small}})\to A_{\Zz}(M,\Gg_{\mathrm{big}}),\qquad
\psi_{\Zz}(x_H)=x_H+x_{F_{\mathrm{new}}}\ (H\in B),\quad \psi_{\Zz}(x_H)=x_H\ (H\notin B),
\]
is injective with saturated image in each graded degree.
\end{lem}

\begin{proof}
Adjoining $F_{\mathrm{new}}$ is the stellar subdivision of the Feichtner-Yuzvinsky fan along the cone indexed by $B=\Factor(F_{\mathrm{new}},\Gg_{\mathrm{small}})$, and $\psi_{\Zz}$ is the corresponding Chow pullback \cite{FY04}. By the integral blowup formula for Chow groups applied to this smooth stellar subdivision, one has degreewise
\[
A^p_{\Zz}(M,\Gg_{\mathrm{big}})\simeq\psi_{\Zz}\,A^p_{\Zz}(M,\Gg_{\mathrm{small}})\oplus\bigoplus_{a=1}^{c-1}A^{p-a}_{\Zz}(Z),
\]
where $c$ is the codimension of the stellar-subdivision stratum and $A_{\Zz}(Z)$ is the Chow ring of that stratum. By the Feichtner-Yuzvinsky description of a stellar subdivision, $A_{\Zz}(Z)$ is the tensor product of the factor Chow rings $A_{\Zz}(M|F_i,\Gg_{\mathrm{small}}|F_i)$ and $A_{\Zz}(M/F_{\mathrm{new}},\Gg_{\mathrm{small}}/F_{\mathrm{new}})$, which is free over $\Zz$ on the products of standard monomial bases \cite{FY04}. Hence each exceptional summand $A^{p-a}_{\Zz}(Z)$ is free, so the cokernel of $\psi_{\Zz}^p$ is free and $\psi_{\Zz}^p$ is injective with saturated image.
\end{proof}

\begin{lem}[Filtered saturation]\label{lem:filtered-saturation}
Let $f\colon A\to B$ be a homomorphism of finitely generated free abelian groups, filtered for finite separated filtrations, such that $\operatorname{gr}f$ is injective with saturated image in each degree. Then $f$ has saturated image.
\end{lem}

\begin{proof}
The filtration is decreasing, finite, and separated, $A=F^0A\supseteq F^1A\supseteq\cdots\supseteq 0$, with each $\operatorname{gr}^p f$ injective with saturated image. Suppose $nb=f(a)$. We construct $u_p\in F^pA$ inductively. Assume $b_p:=b-f(u_0+\cdots+u_{p-1})\in F^pB$ and $nb_p=f(a_p)$ for some $a_p\in F^pA$. In $\operatorname{gr}^p B$ one has $n\overline{b_p}=\operatorname{gr}^p f(\overline{a_p})$, so by saturation of $\operatorname{gr}^p f$ there is $\overline{u_p}\in\operatorname{gr}^p A$ with $\operatorname{gr}^p f(\overline{u_p})=\overline{b_p}$. Choose a representative $u_p\in F^pA$. Then $b_{p+1}:=b_p-f(u_p)\in F^{p+1}B$ and $nb_{p+1}=f(a_p-nu_p)$; injectivity of $\operatorname{gr}^p f$ gives $a_p-nu_p\in F^{p+1}A$, so the induction continues. Finiteness gives $b=f(u_0+\cdots+u_N)\in f(A)$.
\end{proof}

\begin{prop}[Free cokernel of the one-step map]\label{prop:one-step-saturation}
In the one-step situation of Set-up~\ref{setup:one-step}, the homomorphism
\[
\phi_K\colon K_{\Zz}(M,\Gg_{\mathrm{small}})\longrightarrow K_{\Zz}(M,\Gg_{\mathrm{big}})
\]
is a split injection of abelian groups with free cokernel. Equivalently, the integer matrix of $\phi_K$ in the standard $\tau$-monomial $\Zz$-bases has Smith normal form with every nonzero invariant factor equal to $1$. In particular $\phi_K$ has saturated image, and the same conclusions hold for the composite along any chain of one-flat enlargements.
\end{prop}

\begin{proof}
The map $\phi_K$ is the well-defined $\Zz$-algebra homomorphism of Set-up~\ref{setup:one-step}, with the explicit generator rule
\[
\phi_K(\tau_H)=\tau_H\quad(H\notin B),\qquad
\phi_K(\tau_H)=\tau_H+\tau_{F_{\mathrm{new}}}-\tau_H\tau_{F_{\mathrm{new}}}\quad(H\in B),
\]
where $B=\Factor(F_{\mathrm{new}},\Gg_{\mathrm{small}})$. It is injective, and both rings are finitely generated free abelian groups with their standard $\tau$-monomial $\Zz$-bases by Proposition~\ref{prop:basis}.

We first record that $\phi_K$ has saturated image: if $y\in K_{\Zz}(M,\Gg_{\mathrm{big}})$ satisfies $y\otimes 1\in\im(\phi_K\otimes_{\Zz}\Qq)$, then $y\in\im(\phi_K)$. By Lemma~\ref{lem:assoc-graded} the $\tau$-adic associated graded of $K_{\Zz}$ is the integral Feichtner-Yuzvinsky Chow ring, and on associated graded $\phi_K$ induces the homomorphism $\psi_{\Zz}$ of Lemma~\ref{lem:chow-comparison}, which is injective with saturated image in each graded degree. Lemma~\ref{lem:filtered-saturation} then gives that $\phi_K$ has saturated image. This argument is intrinsic and uses no complete-fan, toric-center, or auxiliary-completion hypothesis.

Now let $C:=K_{\Zz}(M,\Gg_{\mathrm{big}})/\phi_K(K_{\Zz}(M,\Gg_{\mathrm{small}}))$. It is finitely generated. If $y\in K_{\Zz}(M,\Gg_{\mathrm{big}})$ and $n\geq 1$ satisfy $n\cdot[y]=0$ in $C$, then $n y=\phi_K(x)$ for some $x$, so $y\otimes 1=(\phi_K\otimes\Qq)\bigl((x\otimes 1)/n\bigr)\in\im(\phi_K\otimes\Qq)$; saturation gives $y\in\im(\phi_K)$, i.e. $[y]=0$. Hence $C$ is torsion-free, and a finitely generated torsion-free abelian group is free. The short exact sequence
\[
0\to K_{\Zz}(M,\Gg_{\mathrm{small}})\xrightarrow{\ \phi_K\ }K_{\Zz}(M,\Gg_{\mathrm{big}})\to C\to 0
\]
therefore splits, so $\phi_K$ is a split injection. Choosing the standard $\tau$-monomial basis of $K_{\Zz}(M,\Gg_{\mathrm{small}})$ and combining its image with a lifted $\Zz$-basis of $C$ gives a basis of $K_{\Zz}(M,\Gg_{\mathrm{big}})$ in which the matrix of $\phi_K$ is a block column with an identity block and a zero block; since a change to the standard $\tau$-monomial bases is by unimodular matrices, the Smith normal form has every nonzero invariant factor equal to $1$. A split injection with free cokernel has saturated image. For a chain $\Gg_0\subset\cdots\subset\Gg_m$ the composite of split injections with free cokernels is again one.
\end{proof}

Feichtner and Yuzvinsky construct the smooth, generally noncomplete toric variety $X(\Sigma_{M,\Gg})$ from the atomic lattice and building set, and identify its Chow ring with $A_{\Zz}(M,\Gg)$ \cite{FY04}; the same one-step map is then induced by a toric pullback. We do not use any such model in the proof above, which is intrinsic to $K_{\Zz}(M,\Gg)$.

\begin{prop}[Upper-set line classes]\label{prop:upper-set}
Let $Y\in \Gg_{\mathrm{small}}$. Along a one-step chain, the class
\[
\Lambda_Y^{\Gg}:=\prod_{\substack{H\in \Gg\\ Y\leq H}}(1-\tau_H)^{-1}
\]
is functorial:
\[
\phi_K(\Lambda_Y^{\Gg_{\mathrm{small}}})=\Lambda_Y^{\Gg_{\mathrm{big}}}.
\]
\end{prop}

\begin{proof}
Equivalently, write $\eta_Y^{\Gg}:=1-\Lambda_Y^{\Gg}$. The one-step generator formula gives $\phi_K(\eta_Y^{\Gg_{\mathrm{small}}})=\eta_Y^{\Gg_{\mathrm{big}}}$ by multiplying the factors indexed by flats above $Y$ and using the factorization of $F_{\mathrm{new}}$ by its maximal old factors. Subtracting from $1$ gives the displayed formula for $\Lambda_Y$.
\end{proof}

\section{Realizable comparison and generator normalization}\label{sec:theta}

In this section, we prove the realizable statement in Theorem~\ref{thm:main-integral}. The main issue is the integral normalization of the one-step comparison.

\begin{setup}\label{setup:theta-step}
Let $M$ be realized over $\mathbb C$ by a linear subspace $L$ not contained in any coordinate hyperplane. Let $\Gg_{\mathrm{small}}\subset \Gg_{\mathrm{big}}=\Gg_{\mathrm{small}}\cup\{F_{\mathrm{new}}\}$ be a one-step refinement, and let
\[
\pi\colon W_{L,\Gg_{\mathrm{big}}}\longrightarrow W_{L,\Gg_{\mathrm{small}}}
\]
be the corresponding wonderful-model blowup. Put
\[
B:=\Factor(F_{\mathrm{new}},\Gg_{\mathrm{small}}).
\]
\end{setup}

\begin{prop}[Generator normalization]\label{prop:generator-normalization}
In Set-up~\ref{setup:theta-step}, the integral one-step map satisfies
\[
\phi_K(\tau_H)=\tau_H
\]
for $H\notin B$, and
\[
\phi_K(\tau_H)=\tau_H+\tau_{F_{\mathrm{new}}}-\tau_H\tau_{F_{\mathrm{new}}}
\]
for $H\in B$. Moreover $\phi_K(\tau_E)=\tau_E$.
\end{prop}

\begin{proof}
The first two formulas are the definition of the one-step integral K-ring map. Since every element of $B$ is contained in the proper flat $F_{\mathrm{new}}$, the top flat $E$ is not an element of $B$. Substituting $H=E$ gives $\phi_K(\tau_E)=\tau_E$.
\end{proof}
\begin{humancomment}
The AI system reproves this statement by the same method; it can be shortened
by observing that the Chow pullback for the toric models agrees with the Chow
pullback for the wonderful models.
\end{humancomment}
\begin{lem}[Wonderful center product and torsion-freeness]\label{lem:wonderful-center}
Let $M$ be realized by $L$. In the one-step wonderful-model blowup $\pi\colon W_{L,\Gg_{\mathrm{big}}}\to W_{L,\Gg_{\mathrm{small}}}$, the center $Z$ is canonically isomorphic to a product of wonderful models attached to the restriction factors $M|F_i$ for $F_i\in B$ and to the contraction factor $M/F_{\mathrm{new}}$. In particular $K_0(Z)$ is a free abelian group.
\end{lem}

\begin{proof}
The one-step enlargement is a De Concini-Procesi blowup whose center is the intersection of the boundary divisors indexed by $B$; that intersection is canonically the product of the wonderful models of the restriction factors $M|F_i$ and the contraction factor $M/F_{\mathrm{new}}$ \cite{DCP95}. We prove slightly more generally, by simultaneous induction on the De Concini-Procesi blowup construction, that every finite product of such wonderful models has free $K_0$. The base case is a finite product of projective spaces, whose $K_0$ is free. Suppose one factor $X'$ is obtained from $X$ by blowing up a smooth center $C$, and let $Y$ be the product of the remaining factors. Since blowup commutes with flat base change, $X'\times Y\simeq\operatorname{Bl}_{C\times Y}(X\times Y)$, and the K-theoretic blowup formula \cite{Tho93} gives
\[
K_0(X'\times Y)\simeq K_0(X\times Y)\oplus\bigoplus_{a=1}^{c-1}K_0(C\times Y),\qquad c=\operatorname{codim}(C,X).
\]
The groups on the right are free by induction, since $C$ is a product of smaller wonderful models. Hence every finite product of wonderful models occurring in the construction has free $K_0$, and in particular $K_0(Z)$ is free.
\end{proof}

\begin{lem}[Geometric pullback saturation]\label{lem:geometric-saturation}
Let $M$ be realized by $L$, and let $\pi\colon W_{L,\Gg_{\mathrm{big}}}\to W_{L,\Gg_{\mathrm{small}}}$ be the one-step wonderful-model blowup along the smooth center $Z$ of Lemma~\ref{lem:wonderful-center}. Then the pullback $\pi^*\colon K_0(W_{L,\Gg_{\mathrm{small}}})\to K_0(W_{L,\Gg_{\mathrm{big}}})$ is a split injection with free cokernel; in particular its image is saturated.
\end{lem}

\begin{proof}
The blowup of a smooth variety along a smooth center has the K-theoretic decomposition \cite{Tho93}
\[
K_0(W_{L,\Gg_{\mathrm{big}}})=\pi^*K_0(W_{L,\Gg_{\mathrm{small}}})\oplus\bigoplus_{a=1}^{c-1}j_*\bigl(p^*K_0(Z)\,t^a\bigr),
\]
where $j$ is the exceptional inclusion, $p$ the projective-bundle projection, $t=[\Oo(-1)]$, and $c$ the codimension. Each exceptional summand is a twist of $K_0(Z)$, which is free by Lemma~\ref{lem:wonderful-center}. Hence $\pi^*$ is a split injection with free cokernel, so its image is saturated.
\end{proof}

\begin{prop}[Integral theta descent]\label{prop:theta-descent}
Here $A_{\mathrm{small}}$ and $A_{\mathrm{big}}=\pi^*A_{\mathrm{small}}$ denote the hyperplane line bundles on the wonderful models $W_{L,\Gg_{\mathrm{small}}}$ and $W_{L,\Gg_{\mathrm{big}}}$, so that the top generator corresponds to $1-[A]$ under the atom-character relation. Assume that the refined comparison
\[
\theta_{\mathrm{big}}\colon K_{\Zz}(M,\Gg_{\mathrm{big}})\longrightarrow K_0(W_{L,\Gg_{\mathrm{big}}})
\]
sends every proper-flat generator $\tau_H$ to $[\Oo_{D_H}]$ and sends the top generator $\tau_E$ to $1-[A_{\mathrm{big}}]$. Then there exists a unique integral unital ring isomorphism
\[
\theta_{\mathrm{small}}\colon K_{\Zz}(M,\Gg_{\mathrm{small}})\longrightarrow K_0(W_{L,\Gg_{\mathrm{small}}})
\]
such that
\[
\pi^*(\theta_{\mathrm{small}}(\xi))=\theta_{\mathrm{big}}(\phi_K(\xi))
\]
for every $\xi\in K_{\Zz}(M,\Gg_{\mathrm{small}})$. It satisfies
\[
\theta_{\mathrm{small}}\left((1-\tau_H)^{-1}\right)=[\Oo(D_H)]
\]
for every proper flat $H\in \Gg_{\mathrm{small}}$.
\end{prop}

\begin{proof}
By Proposition~\ref{prop:one-step-saturation} the combinatorial image of $\phi_K$ is saturated, and by Lemma~\ref{lem:geometric-saturation} the geometric pullback image of $\pi^*$ is saturated. Via $\theta_{\mathrm{big}}$, the subgroups $\theta_{\mathrm{big}}(\phi_K K_{\Zz}(M,\Gg_{\mathrm{small}}))$ and $\pi^*K_0(W_{L,\Gg_{\mathrm{small}}})$ are saturated subgroups of the free abelian group $K_0(W_{L,\Gg_{\mathrm{big}}})$, and the rational comparison identifies their $\Qq$-spans. Two saturated subgroups of a free abelian group with the same $\Qq$-span are equal, so
\[
\theta_{\mathrm{big}}\bigl(\phi_K K_{\Zz}(M,\Gg_{\mathrm{small}})\bigr)=\pi^*K_0(W_{L,\Gg_{\mathrm{small}}}).
\]
This identification upgrades the rational descent to an integral descent, giving a unique $\theta_{\mathrm{small}}$.

The generator values follow from the geometric generator compatibility. If $H\notin B$, the boundary divisor pulls back to the boundary divisor. If $H\in B$, then the total transform has K-class
\[
[\Oo_{D_H}]+[\Oo_{D_{F_{\mathrm{new}}}}]-[\Oo_{D_H}]\,[\Oo_{D_{F_{\mathrm{new}}}}],
\]
which is exactly the image of $\tau_H+\tau_{F_{\mathrm{new}}}-\tau_H\tau_{F_{\mathrm{new}}}$. The top generator gives $1-[A_{\mathrm{small}}]$ after descent. Inverting $1-\tau_H$ gives the line bundle class $[\Oo(D_H)]$.
\end{proof}

\begin{proof}[Proof of Theorem~\ref{thm:main-integral}, realizable clause]
Start at the maximal building set. The Larson-Li-Payne-Proudfoot comparison identifies the intrinsic integral K-ring with the K-ring of the maximal wonderful model \cite{LLPP24}. Along any chain from the maximal building set down to $\Gg$, Proposition~\ref{prop:theta-descent} descends the comparison integrally. Proposition~\ref{prop:generator-normalization} records the generator normalization at every one-step descent.

Thus we obtain an integral unital ring isomorphism
\[
\theta_{\Gg}^{\Zz}\colon K_{\Zz}(M,\Gg)\to K_0(W_{L,\Gg})
\]
with
\[
\theta_{\Gg}^{\Zz}\left((1-\tau_F)^{-1}\right)=[\Oo_{W_{L,\Gg}}(D_F)]
\]
for every $F\in \Gg^\circ$. Both $\theta_{\Gg}^{\Zz}(T_{M,\Gg}^{\Zz})$ and $[T_{W_{L,\Gg}}]$ rationalize to the same class: after tensoring with $\Qq$ this is the realizable specialization $\theta(T_{M,\Gg})=[T_{W_{L,\Gg}}]$ of Proposition~\ref{prop:quotient-tangent} (Step~6), since $\rho_K(T_{M,\Gg}^{\Zz})=T_{M,\Gg}$ by Proposition~\ref{prop:linkage} and the quotient representative was descended compatibly with the maximal Berget-Eur-Spink-Tseng quotient class. The Grothendieck group $K_0(W_{L,\Gg})$ of the smooth projective wonderful model is torsion-free, so two of its elements with equal rationalizations are equal. Hence
\[
\theta_{\Gg}^{\Zz}(T_{M,\Gg}^{\Zz})=[T_{W_{L,\Gg}}].
\]
\end{proof}

\section{The Hilbert identity}\label{sec:hilbert-identity}

The goal of this section is to prove the \(P^K=\Hilb\) part of the rational tangent-class theorem. The proof first treats the maximal building set and then propagates the identity through one-flat refinements.

\begin{defn}\label{defn:pk}
Let \((N,\Hh)\) be a top-containing pair. Write
\[
C_Q(N,\Hh;u)=\prod_{a=1}^{q_Q(N,\Hh)}(1+y_a u)
\]
after passing to a splitting algebra, where \(q_Q(N,\Hh)=|\Hh\setminus\{\Top(N)\}|-(\rank(N)-1)\). Put
\[
g_z(v):=(1-z\exp(-v))\frac{v}{1-\exp(-v)},\qquad g_z(0)=1-z.
\]
The signed K-theoretic Todd polynomial is
\[
P^K(N,\Hh;z):=
\deg_{N,\Hh}
\left[
\prod_{Y\in\Hh\setminus\{\Top(N)\}}g_z(x_Y)
\prod_{a=1}^{q_Q(N,\Hh)}g_z(y_a)^{-1}
\right]_{\rank(N)-1}.
\]
Equivalently,
\[
P^K(N,\Hh;z)
=
\deg_{N,\Hh}\left(\ch(\lambda_{-z}(T_{N,\Hh}^{\vee}))\td(T_{N,\Hh})\right).
\]
\end{defn}

\begin{lem}[One-step quotient integrands]\label{lem:one-step-quotient-integrands}
Let \(M\), \(\Gg_{\mathrm{small}}\), \(\Gg_{\mathrm{big}}\), and \(F\) be as in Proposition~\ref{prop:k-recursion}, and let
\[
\psi\colon A_{\Qq}(M,\Gg_{\mathrm{small}})\longrightarrow
A_{\Qq}(M,\Gg_{\mathrm{big}})
\]
be the Chow pullback for the one-flat enlargement. Then
\[
C_Q(M,\Gg_{\mathrm{big}};u)=\psi(C_Q(M,\Gg_{\mathrm{small}};u)).
\]
Consequently, if
\[
C_Q(M,\Gg_{\mathrm{small}};u)=\prod_{a=1}^{q}(1+y_a u)
\]
in a splitting algebra, then a splitting-root presentation for
\(C_Q(M,\Gg_{\mathrm{big}};u)\) is
\[
\prod_{a=1}^{q}(1+\psi(y_a)u)\cdot(1+0\cdot u).
\]
Moreover, for every top-containing pair \((N,\Hh)\) appearing in the one-flat recursion, \(P^K(N,\Hh;z)\) is the same as the Chow splitting-root expression
\[
P_Q(N,\Hh;z):=
\deg_{N,\Hh}
\left[
\prod_{Y\in\Hh\setminus\{\Top(N)\}}g_z(x_Y)
\prod_{a=1}^{q_Q(N,\Hh)}g_z(y_a)^{-1}
\right]_{\rank(N)-1}.
\]
\end{lem}

\begin{proof}
The descended quotient Chern polynomial \(C_Q\) is characterized by pullback to the maximal building set. Since the refinement from \(\Gg_{\mathrm{small}}\) to the maximal building set factors through \(\Gg_{\mathrm{big}}\), injectivity of Chow pullback gives the one-step compatibility
\(C_Q(M,\Gg_{\mathrm{big}};u)=\psi(C_Q(M,\Gg_{\mathrm{small}};u))\).
The integer \(q_Q\) increases by one when the single proper flat \(F\) is added, while the rank of \(M\) is unchanged. Thus the additional quotient root may be taken to be \(0\).

The last assertion is the splitting-root convention in Definition~\ref{defn:pk}: the boundary line summands \((1-\tau_Y)^{-1}\) in \(T_{N,\Hh}\) contribute the factors \(g_z(x_Y)\), while the quotient class \(Q_{N,\Hh}\), with roots \(y_a\), contributes the inverse factors \(g_z(y_a)^{-1}\). The expression is independent of the chosen splitting roots because the Hirzebruch product is determined by the virtual total Chern series and the virtual rank.
\end{proof}
\begin{humancomment}
By definition, $P^K = P_Q$.
\end{humancomment}
\begin{lem}[Exceptional trace]\label{lem:exceptional-trace}
Keep the notation of Lemma~\ref{lem:one-step-quotient-integrands}, and write
\[
\Fact_{\Gg_{\mathrm{small}}}(F)=\{F_1,\dots,F_\ell\}.
\]
Then \(\ell\ge 2\). Put \(e=x_F\in A^1_{\Qq}(M,\Gg_{\mathrm{big}})\), and set
\[
I_{\mathrm{small}}(z):=
\prod_{Y\in\Gg_{\mathrm{small}}\setminus\{E\}}g_z(x_Y)
\prod_{a=1}^{q_Q(M,\Gg_{\mathrm{small}})}g_z(y_a)^{-1}.
\]
With the quotient roots of \(\Gg_{\mathrm{big}}\) chosen as in Lemma~\ref{lem:one-step-quotient-integrands}, the big integrand is
\[
I_{\mathrm{big}}(z)=\psi(I_{\mathrm{small}}(z))R_F(z),
\]
where
\[
R_F(z)=
\frac{g_z(e)}{1-z}
\prod_{i=1}^{\ell}\frac{g_z(x_{F_i})}{g_z(x_{F_i}+e)}.
\]
Therefore
\[
P_Q(M,\Gg_{\mathrm{big}};z)
=P_Q(M,\Gg_{\mathrm{small}};z)+\operatorname{Exc}_F(z),
\]
where
\[
\operatorname{Exc}_F(z):=
\deg_{M,\Gg_{\mathrm{big}}}
\left[\psi(I_{\mathrm{small}}(z))(R_F(z)-1)\right]_{\rank(M)-1}.
\]
If \(Z_F\) is the one-flat center and
\[
i^*\colon A_{\Qq}(M,\Gg_{\mathrm{small}})\to A^\bullet(Z_F)_{\Qq}
\]
is the center restriction, then
\[
\operatorname{Exc}_F(z)
=(z+z^2+\cdots+z^{\ell-1})
\deg_{Z_F}\left([\theta_Z(z)]_{\rank(M)-1-\ell}\right),
\]
where
\[
\theta_Z(z):=
i^*(I_{\mathrm{small}}(z))
\prod_{i=1}^{\ell}g_z(i^*x_{F_i})^{-1}.
\]
\end{lem}

\begin{proof}
The inequality \(\ell\ge 2\) follows from the building-set axiom. Indeed, the factors are the maximal old building-set elements contained in \(F\); if there were only one such factor, its union would be \(F\), so \(F\) would already lie in \(\Gg_{\mathrm{small}}\).

The formula for \(I_{\mathrm{big}}\) is obtained by comparing the boundary and quotient roots after the one-flat pullback. The new boundary root contributes \(g_z(e)\), and the new zero quotient root contributes \(g_z(0)^{-1}=(1-z)^{-1}\). For each old factor \(F_i\), the pullback of the old boundary class is \(x_{F_i}+e\), while the boundary class in the big model is \(x_{F_i}\), giving the ratio
\(g_z(x_{F_i})/g_z(x_{F_i}+e)\). This gives the displayed correction factor \(R_F(z)\).

Taking the \((\rank(M)-1)\)-degree trace, the term \(\psi(I_{\mathrm{small}}(z))\) has the same trace as \(I_{\mathrm{small}}(z)\), because the normalized top-degree trace is preserved by the one-step Chow pullback. The remaining term is \(\operatorname{Exc}_F(z)\).

It remains to identify this exceptional term. Let
\(\pi\colon A_{\Qq}(M,\Gg_{\mathrm{big}})\to A_{\Qq}(M,\Gg_{\mathrm{small}})\)
be the Chow pushforward and \(i_*\) the Gysin map from the center. The one-flat exceptional pushforward calculation gives
\[
\pi_*\bigl(\psi(I_{\mathrm{small}}(z))(R_F(z)-1)\bigr)
=(z+z^2+\cdots+z^{\ell-1})\,i_*(\theta_Z(z)).
\]
The top trace is compatible with \(\pi_*\), and the center projection formula identifies
\(\deg_{M,\Gg_{\mathrm{small}}}(i_*u)\) with \(\deg_{Z_F}(u)\) in top center degree. Taking homogeneous degree \(\rank(M)-1\) proves the formula for \(\operatorname{Exc}_F(z)\).
\end{proof}

\begin{lem}[Center factorization]\label{lem:center-factorization}
With the notation of Lemma~\ref{lem:exceptional-trace}, there is a canonical graded product identification
\[
A^\bullet(Z_F)_{\Qq}\cong
\left(\bigotimes_{i=1}^{\ell}A_{\Qq}(M|F_i,\Gg_{\mathrm{small}}|F_i)\right)
\otimes
A_{\Qq}(M/F,\Gg_{\mathrm{small}}/F).
\]
Under this identification, the top trace on \(Z_F\) is the product of the top traces on the factors, and
\[
\theta_Z(z)=
\left(\prod_{i=1}^{\ell}I_i(z)\right)I_{\mathrm{con}}(z),
\]
where \(I_i(z)\) is the splitting-root integrand defining
\(P_Q(M|F_i,\Gg_{\mathrm{small}}|F_i;z)\), and \(I_{\mathrm{con}}(z)\) is the splitting-root integrand defining
\(P_Q(M/F,\Gg_{\mathrm{small}}/F;z)\).
\end{lem}

\begin{proof}
The center \(Z_F\) is the orbit closure of the cone generated by the rays indexed by \(F_1,\dots,F_\ell\) in the nested fan for \(\Gg_{\mathrm{small}}\). Its closed star is canonically the product of the nested fans for the induced restriction pairs \((M|F_i,\Gg_{\mathrm{small}}|F_i)\) and the induced contraction pair \((M/F,\Gg_{\mathrm{small}}/F)\). Passing to Chow rings gives the displayed tensor-product identification. The top degree is
\[
\sum_{i=1}^{\ell}(\rank(F_i)-1)+(\rank(M)-\rank(F)-1)
=\rank(M)-1-\ell,
\]
and the degree of a top simple tensor is the product of the factor degrees.

We now compare the integrands. First, after the normal factors indexed by \(F_1,\dots,F_\ell\) are removed, the boundary Chern polynomial on the center becomes the product of the boundary Chern polynomials of the restriction factors and the contraction factor. Second, the quotient polynomial restricts to the product
\[
i^*C_Q(M,\Gg_{\mathrm{small}};u)
=
\left(\prod_{i=1}^{\ell}C_Q(M|F_i,\Gg_{\mathrm{small}}|F_i;u)\right)
\cdot C_Q(M/F,\Gg_{\mathrm{small}}/F;u)
\]
under the same product identification, where the displayed multiplication is in the tensor product of the factor Chow rings. Equivalently, after passing to a common splitting algebra, the virtual total Chern series of the center presentation is the product of the virtual total Chern series of the induced factor presentations. The virtual ranks agree by the top-degree calculation above. Since the \(g_z\)-Hirzebruch product depends only on the virtual total Chern series and virtual rank, the center integrand is exactly the product of the restriction and contraction integrands.
\end{proof}
\begin{humancomment}
The fact that the quotient polynomial restricts to the product is implied in the proof of Proposition~\ref{prop:quotient-tangent}. This corresponds to \cite[Proposition~4.10]{Cheng26}.
\end{humancomment}
\begin{prop}[One-flat K-theoretic recursion]\label{prop:k-recursion}
Let \(M\) be a loopless matroid on \(E\), and let
\[
\Gg_{\mathrm{big}}=\Gg_{\mathrm{small}}\cup\{F\}
\]
be a one-flat enlargement of top-containing Feichtner--Yuzvinsky building sets, with \(F\notin \Gg_{\mathrm{small}}\) a nonempty proper flat. Let
\[
\Fact_{\Gg_{\mathrm{small}}}(F)=\{F_1,\dots,F_{\ell}\}.
\]
Then
\[
\begin{aligned}
P^K(M,\Gg_{\mathrm{big}};z)
&=
P^K(M,\Gg_{\mathrm{small}};z) \\
&\quad+
(z+z^2+\cdots+z^{\ell-1})
\left(\prod_{i=1}^{\ell}P^K(M|F_i,\Gg_{\mathrm{small}}|F_i;z)\right)
P^K(M/F,\Gg_{\mathrm{small}}/F;z).
\end{aligned}
\]
\end{prop}

\begin{proof}
By Lemma~\ref{lem:one-step-quotient-integrands}, \(P^K=P_Q\) for \((M,\Gg_{\mathrm{small}})\), \((M,\Gg_{\mathrm{big}})\), every restriction factor \((M|F_i,\Gg_{\mathrm{small}}|F_i)\), and the contraction factor \((M/F,\Gg_{\mathrm{small}}/F)\). It is enough to prove the formula for \(P_Q\).

Lemma~\ref{lem:exceptional-trace} gives
\[
P_Q(M,\Gg_{\mathrm{big}};z)
=
P_Q(M,\Gg_{\mathrm{small}};z)
+(z+z^2+\cdots+z^{\ell-1})
\deg_{Z_F}\left([\theta_Z(z)]_{\rank(M)-1-\ell}\right).
\]
By Lemma~\ref{lem:center-factorization}, the center integrand is the product of the induced restriction and contraction integrands, and the center trace is the product trace. Hence
\[
\deg_{Z_F}\left([\theta_Z(z)]_{\rank(M)-1-\ell}\right)
=
\left(\prod_{i=1}^{\ell}P_Q(M|F_i,\Gg_{\mathrm{small}}|F_i;z)\right)
P_Q(M/F,\Gg_{\mathrm{small}}/F;z).
\]
Substitution gives the one-flat recursion for \(P_Q\). Replacing each \(P_Q\) by the equal \(P^K\) gives the stated formula. The preceding exceptional-trace lemma also proves that a genuine one-flat enlargement has \(\ell\ge 2\).
\end{proof}

\begin{prop}[One-flat Hilbert recursion]\label{prop:hilbert-recursion}
Let \(M\), \(\Gg_{\mathrm{small}}\), \(\Gg_{\mathrm{big}}\), \(F\), and \(\ell\) be as in Proposition~\ref{prop:k-recursion}. Then
\[
\begin{aligned}
\Hilb(M,\Gg_{\mathrm{big}};z)
&=
\Hilb(M,\Gg_{\mathrm{small}};z) \\
&\quad+
(z+z^2+\cdots+z^{\ell-1})
\left(\prod_{i=1}^{\ell}
\Hilb(M|F_i,\Gg_{\mathrm{small}}|F_i;z)\right)
\Hilb(M/F,\Gg_{\mathrm{small}}/F;z).
\end{aligned}
\]
\end{prop}

\begin{proof}
The one-step Hilbert-series recursion of Eur--Ferroni--Matherne--Pagaria--Vecchi \cite{EFMPV25} gives
\[
\begin{aligned}
\Hilb(M,\Gg_{\mathrm{big}};z)
&=
\Hilb(M,\Gg_{\mathrm{small}};z) \\
&\quad+
(z+z^2+\cdots+z^{\ell-1})
\Hilb(M|F,\Gg_{\mathrm{small}}|F;z)
\Hilb(M/F,\Gg_{\mathrm{small}}/F;z).
\end{aligned}
\]
Here \(\ell\) is the number of maximal elements of
\(\Gg_{\mathrm{small}}\) contained in \(F\), and the restriction and
contraction building sets are the induced ones.

It remains only to rewrite the restriction factor in the form used by the
center product. The factors \(F_1,\dots,F_\ell\) give the decomposition of
the interval below \(F\), hence \(M|F\) is the direct sum of the restrictions
\(M|F_i\). The induced restriction building set \(\Gg_{\mathrm{small}}|F\)
has maximal elements \(F_1,\dots,F_\ell\), and its Feichtner--Yuzvinsky Chow
algebra is the tensor product of the Chow algebras
\[
A_{\Qq}(M|F_i,\Gg_{\mathrm{small}}|F_i),\qquad 1\leq i\leq \ell .
\]
Taking graded dimensions gives
\[
\Hilb(M|F,\Gg_{\mathrm{small}}|F;z)
=
\prod_{i=1}^{\ell}\Hilb(M|F_i,\Gg_{\mathrm{small}}|F_i;z).
\]
Substituting this factorization into the one-step recursion proves the
displayed formula.
\end{proof}

\begin{prop}[Maximal endpoint]\label{prop:maximal-endpoint}
Let \(R\) be a loopless matroid with finite lattice of flats, bottom flat \(\emptyset\), top flat \(\Top(R)\), and rank function \(\rk_R\). Let
\[
\Gg_{\max}(R):=L(R)\setminus\{\emptyset\}.
\]
Then
\[
P^K(R,\Gg_{\max}(R);z)=\Hilb(R,\Gg_{\max}(R);z),
\]
and both sides are equal to
\[
\sum_{\emptyset=F_0<F_1<\cdots<F_m}
\prod_{i=1}^m
z\frac{1-z^{\rk_R(F_i)-\rk_R(F_{i-1})-1}}{1-z},
\]
where the sum ranges over all chains of flats starting at \(\emptyset\), including the chain with only \(\emptyset\), whose summand is \(1\).
\end{prop}

\begin{proof}
Put \(\Gg=\Gg_{\max}(R)\). Cheng's preprint
\cite[Definition~3.1]{Cheng25} defines, for every loopless matroid \(R\), a
tangent K-class \(T_R^{\mathrm{Ch}}\) on the maximal matroid toric model by
restricting \(T_{X_E}-Q_R\) from the permutohedral variety, where \(Q_R\) is the
BEST tautological quotient class. The displayed tangent formula before
\cite[Theorem~3.4]{Cheng25} gives the boundary contribution
\(\prod_F(1+x_F)\), and \cite[Theorem~3.4]{Cheng25} identifies the quotient
Chern factor. By Proposition~\ref{prop:quotient-tangent} in the maximal case,
the MatTan class \(T_{R,\Gg}\) is reconstructed from the same BEST quotient
Chern polynomial. Its virtual rank is \(|\Gg\setminus\{\Top(R)\}|-q_Q(R,\Gg)=\rk_R(\Top(R))-1\), since \(q_Q(R,\Gg)=|\Gg\setminus\{\Top(R)\}|-(\rk_R(\Top(R))-1)\); this equals the rank of \(T_R^{\mathrm{Ch}}\), the tangent bundle of the maximal wonderful model of dimension \(\rk_R(\Top(R))-1\). Two rational K-classes of equal virtual rank and equal total Chern class have the same Chern character, and therefore agree
under the rational Chern-character isomorphism. Equivalently, they have the same
rank and the same total Chern class
\[
c_z(T)=c_z(Q)^{-1}\prod_{F\in\Gg\setminus\{\Top(R)\}}(1+z x_F),
\]
with the homogeneous \(z\)-grading convention used in this paper.

Cheng's Chow-polynomial formula \cite[Theorem~1.1(4)]{Cheng25} gives
\[
\dim_{\Qq}A_{\Qq}(R,\Gg)^p
=
(-1)^p
\deg_{R,\Gg}\left(\ch\left(\bigwedge^p T_{R,\Gg}^{\vee}\right)\td(T_{R,\Gg})\right)
\]
for every \(p\). Therefore the coefficient of \(z^p\) in \(P^K(R,\Gg;z)\) is the coefficient of \(z^p\) in \(\Hilb(R,\Gg;z)\).

Finally, the displayed chain formula for the maximal Hilbert polynomial is the
Ferroni--Matherne--Stevens--Vecchi formula \cite[Proposition~3.5]{FMSV24}.
Combining these two equalities proves the proposition.
\end{proof}

\begin{thm}[Hilbert identity for arbitrary building sets]\label{thm:hilbert-identity}
For every loopless matroid \(M\) and every top-containing Feichtner--Yuzvinsky building set \(\Gg\), one has
\[
P^K(M,\Gg;z)=\Hilb(M,\Gg;z).
\]
Consequently, if \(\rank(M)=d+1\), then for every integer \(i\) with \(0\leq i\leq d\),
\[
\dim_{\Qq}A_{\Qq}(M,\Gg)^i
=
(-1)^i
\deg_{M,\Gg}\left(\ch\left(\bigwedge^i T_{M,\Gg}^{\vee}\right)\td(T_{M,\Gg})\right).
\]
\end{thm}

\begin{proof}
We prove the equality by induction on \(e=\rank(M)-1\). If \(e=0\), then \(M\)
has no nonempty proper flats and the only top-containing building set is
\(\Gg_{\max}(M)\), so the assertion is Proposition~\ref{prop:maximal-endpoint}.

Assume \(e>0\), and assume the theorem is known for all loopless matroids of
smaller rank. Let \(M\) be loopless with \(\rank(M)-1=e\), and let \(\Gg\) be a
top-containing building set. Choose a one-flat refinement chain
\[
\Gg=\Gg_0\subset \Gg_1\subset\cdots\subset \Gg_s=\Gg_{\max}(M),
\qquad
\Gg_j=\Gg_{j-1}\cup\{F_j\}.
\]
Proposition~\ref{prop:maximal-endpoint} gives
\[
P^K(M,\Gg_s;z)=\Hilb(M,\Gg_s;z).
\]
We descend along the chain. Suppose the equality is known for \((M,\Gg_j)\),
where \(\Gg_j=\Gg_{j-1}\cup\{F_j\}\), and write
\(\Fact_{\Gg_{j-1}}(F_j)=\{F_1,\ldots,F_\ell\}\). Each restriction factor
\((M|F_i,\Gg_{j-1}|F_i)\) and the contraction factor
\((M/F_j,\Gg_{j-1}/F_j)\) has smaller rank than \(M\). By the induction
hypothesis, \(P^K=\Hilb\) for all these induced factors. The exceptional terms
in Propositions~\ref{prop:k-recursion} and~\ref{prop:hilbert-recursion} are
therefore equal. Subtracting the two one-flat recursion formulas gives
\[
P^K(M,\Gg_{j-1};z)=\Hilb(M,\Gg_{j-1};z).
\]
Descending from \(j=s\) to \(j=0\) proves the polynomial identity for
\((M,\Gg)\).

The coefficient identity is the equality of the coefficient of \(z^i\) in \(P^K(M,\Gg;z)=\Hilb(M,\Gg;z)\), using Definition~\ref{defn:pk}.
\end{proof}

\section{Nested Segre tails and the Chern-alpha bound}\label{sec:chern-alpha}

The goal of this section is to prove the Chern-alpha lower bound. The proof uses connected flats, nested supports, and a truncated generating-polynomial identity. Every generating identity below is stated in \(\Qq[t]/(t^r)\), or equivalently coefficientwise in degrees \(0\leq k\leq r-1\).

\begin{setup}\label{setup:nested}
Let \(N\) be a connected loopless matroid of rank \(r\) on \(E\). Put \(e:=r-1\), \(\Hh:=\Hh_{\mathrm{conn}}(N)\), and \(\Gg_{\max}:=L(N)\setminus\{\emptyset\}\). Choose an ordering \(X_1,\dots,X_m\) of \(\Gg_{\max}\setminus\Hh\) such that whenever \(X_p<X_q\), one has \(q<p\). Put \(\Gg_s:=\Hh\cup\{X_1,\dots,X_s\}\), and let
\[
\psi\colon A_{\Qq}(N,\Hh)\to A_{\Qq}(N,\Gg_{\max})
\]
be the composite Chow homomorphism along this reverse-containment refinement chain. Let \(\alpha:=-x_E\), and set
\[
\Gamma_N(z):=C_{\mathrm{BEST}}(N;z)^{-1}
\prod_{Y\in\Hh\setminus\{E\}}(1+z\psi(x_Y)).
\]
For an \(\Hh\)-nested subset \(S\subset\Hh\setminus\{E\}\), let \(U(S)\) be the join of all flats in \(S\), with \(U(\emptyset)=\emptyset\). For \(A\in S\), let
\[
A_-^S:=\bigvee_{\substack{B\in S\\ B<A}}B,
\]
with \(A_-^S=\emptyset\) if no such \(B\) exists, and put
\[
q_S(A):=\rank_N(A)-\rank_N(A_-^S),\qquad
\omega(S):=\prod_{A\in S}(q_S(A)-1),
\]
with \(\omega(\emptyset)=1\).
\end{setup}


\begin{lem}[Pulled connected-base Chern-alpha]\label{lem:pulled-connected-base}
With notation as in Set-up~\ref{setup:nested}, for every \(0\leq k\leq e\),
\[
I_k(N,\Hh):=
\deg_{N,\Hh}\left(c_k(T_{N,\Hh})(-x_E)^{e-k}\right)
=
\deg_{N,\Gg_{\max}}\left([\Gamma_N(z)]_k\alpha^{e-k}\right).
\]
On the right, \(\alpha=-x_E\) denotes the pullback of the same top-flat class to
\(A_{\Qq}(N,\Gg_{\max})\).
\end{lem}

\begin{proof}
By Proposition~\ref{prop:quotient-tangent} the tangent class has formal total Chern polynomial
\[
c_z(T_{N,\Hh})=C_Q(N,\Hh;z)^{-1}\prod_{Y\in\Hh\setminus\{E\}}(1+z x_Y),
\]
and the descended quotient Chern polynomial pulls back to the maximal model as \(\psi(C_Q(N,\Hh;z))=C_{\mathrm{BEST}}(N;z)\). Since \(\psi\) is a homomorphism of graded \(\Qq\)-algebras, it commutes with the \(z\)-graded inverse and with finite products, so
\[
\psi(c_z(T_{N,\Hh}))=\psi(C_Q(N,\Hh;z))^{-1}\prod_{Y\in\Hh\setminus\{E\}}(1+z\,\psi(x_Y))=C_{\mathrm{BEST}}(N;z)^{-1}\prod_{Y\in\Hh\setminus\{E\}}(1+z\,\psi(x_Y))=\Gamma_N(z).
\]
Taking the homogeneous Chow-degree-\(k\) part gives \(\psi(c_k(T_{N,\Hh}))=[\Gamma_N(z)]_k\). Fix \(0\le k\le e\). The class \(c_k(T_{N,\Hh})(-x_E)^{e-k}\) is homogeneous of top Chow degree \(e\); since the top flat \(E\) lies in every building set, \(\psi(x_E)=x_E\), so \(\psi\) carries it to \([\Gamma_N(z)]_k\alpha^{e-k}\). Finally \(\psi\) is the pullback along the proper birational toric morphism \(X_{\Gg_{\max}}\to X_{\Hh}\) of smooth varieties, and such a pullback preserves the degree of a top-dimensional class, because the pushforward of the fundamental class is the fundamental class; hence \(\deg_{N,\Gg_{\max}}(\psi(\xi))=\deg_{N,\Hh}(\xi)\) for every top-degree class \(\xi\). Applying this to \(\xi=c_k(T_{N,\Hh})(-x_E)^{e-k}\) yields the displayed equality.
\end{proof}

\begin{lem}[Connected-flag weights]\label{lem:connected-flag-weights}
Let \(N\) be connected and loopless, and put \(\Hh^0:=\Hh_{\mathrm{conn}}(N)\cup\{\emptyset\}\). For \(F\in\Hh^0\), define \(\tau_N(\emptyset)=1\) and
\[
\tau_N(F):=\sum_{\substack{G\in\Hh^0\\G<F}}\tau_N(G)(\rank_N(F)-\rank_N(G)-1)
\]
for nonempty \(F\in\Hh_{\mathrm{conn}}(N)\). Then
\[
\tau_N(F)=
\sum_{\emptyset=F_0<F_1<\cdots<F_m=F}
\prod_{i=1}^m(\rank_N(F_i)-\rank_N(F_{i-1})-1),
\]
where the sum ranges over strict chains in \(\Hh^0\) ending at \(F\). In particular, \(\tau_N(F)\geq 0\).
\end{lem}

\begin{proof}
We argue by induction over the finite poset \(\Hh^0\). The assertion is immediate for \(F=\emptyset\). For nonempty \(F\), every strict chain ending at \(F\) has a unique penultimate element \(G<F\), and appending \(F\) multiplies the weight of a chain ending at \(G\) by \(\rank_N(F)-\rank_N(G)-1\). Summing over \(G\) gives the recursive formula. Each rank-gap-minus-one factor is nonnegative, so \(\tau_N(F)\geq 0\).
\end{proof}

\begin{lem}[Chain and non-chain decomposition]\label{lem:non-chain-decomposition}
For every integer \(k\), define
\[
AN_k(N):=
\sum_S\omega(S)B(r-\rank_N(U(S)),k-\rank_N(U(S))),
\]
where \(S\) ranges over all \(\Hh\)-nested subsets of \(\Hh\setminus\{E\}\). Let \(NC(N)\) be the set of such \(S\) that are not totally ordered by inclusion. Then
\[
\begin{aligned}
AN_k(N)
&=
B(r,k)
+
\sum_{F\in\Hh\setminus\{E\}}\tau_N(F)B(r-\rank_N(F),k-\rank_N(F))\\
&\qquad
+
\sum_{S\in NC(N)}\omega(S)B(r-\rank_N(U(S)),k-\rank_N(U(S))).
\end{aligned}
\]
\end{lem}

\begin{proof}
Partition the nested supports into chains and non-chains. A totally ordered support is uniquely
\[
S=\{F_1<\cdots<F_m\},
\]
including the empty chain when \(m=0\). Its join is \(F_m\), with \(F_m=\emptyset\) in the empty case, and
\[
\omega(S)=\prod_{i=1}^m(\rank_N(F_i)-\rank_N(F_{i-1})-1).
\]
Summing the chain contributions with fixed terminal flat \(F\) gives \(\tau_N(F)\) by Lemma~\ref{lem:connected-flag-weights}. The remaining contributions are exactly those indexed by \(NC(N)\).
\end{proof}

\begin{lem}[Low-degree chain reduction]\label{lem:low-degree-firewall}
With notation as in Set-up~\ref{setup:nested}, assume \(0\leq k\leq e\) and \(k\leq 3\). Then the non-chain residual in Lemma~\ref{lem:non-chain-decomposition} vanishes:
\[
\sum_{S\in NC(N)}\omega(S)B(r-\rank_N(U(S)),k-\rank_N(U(S)))=0.
\]
Consequently
\[
AN_k(N)=
B(r,k)+
\sum_{F\in\Hh\setminus\{E\}}\tau_N(F)B(r-\rank_N(F),k-\rank_N(F)).
\]
In particular, in degrees \(k\leq 3\), the all-nested coefficient equals the connected-flag \(\tau_N\)-expression.
\end{lem}

\begin{proof}
Fix \(S\in NC(N)\). If \(\omega(S)=0\), its summand is zero. Assume \(\omega(S)\neq 0\). Since \(S\) is not a chain, choose incomparable \(A,B\in S\). Let \(J=A\vee B\). The nestedness condition says \(J\notin\Hh\). Since \(E\in\Hh\), \(J\neq E\), so \(J\) is a proper nonempty flat. Because \(\Hh=\Hh_{\mathrm{conn}}(N)\), the restriction \(N|J\) is disconnected.

Let \(J_1,\dots,J_t\) be the connected components of \(N|J\). A connected flat contained in \(J\) lies in one component. The flats \(A\) and \(B\) cannot lie in the same component: otherwise their join would be contained in that component, contradicting \(A\vee B=J\). Since \(\omega(S)\neq 0\), every local factor \(q_S(C)-1\) for \(C\in S\) is nonzero. Also \(A_-^S<C\) for each \(C\in S\), so \(q_S(C)\geq 1\), and nonvanishing gives \(q_S(C)\geq 2\). Applying this to \(A\) and \(B\), we obtain \(\rank_N(A)\geq 2\) and \(\rank_N(B)\geq 2\). Since rank is additive over different connected components of \(N|J\),
\[
\rank_N(J)\geq \rank_N(A)+\rank_N(B)\geq 4.
\]
Thus \(\rank_N(U(S))\geq 4\). For \(k\leq 3\), the integer \(k-\rank_N(U(S))\) is negative, and hence
\[
B(r-\rank_N(U(S)),k-\rank_N(U(S)))=0.
\]
Every non-chain summand is zero.
\end{proof}

\begin{prop}[Low-degree Gamma formula]\label{prop:low-degree-gamma}
With notation as in Set-up~\ref{setup:nested}, for every integer \(k\) with \(0\leq k\leq e\) and \(k\leq 3\), one has
\[
\deg_{N,\Gg_{\max}}\left([\Gamma_N(z)]_k\alpha^{e-k}\right)=AN_k(N).
\]
Equivalently, in this range,
\[
\deg_{N,\Gg_{\max}}\left([\Gamma_N(z)]_k\alpha^{e-k}\right)
=
B(r,k)+
\sum_{F\in\Hh\setminus\{E\}}\tau_N(F)B(r-\rank_N(F),k-\rank_N(F)).
\]
\end{prop}

\begin{proof}
We first prove the first displayed equality. For \(0\leq k\leq 3\), a nonempty \(\Hh\)-nested support \(S\) with \(\omega(S)\neq0\) and \(\rank_N(U(S))\leq k\) must be a singleton: every element has \(q_S(A)\geq 2\) (so rank \(\geq2\)), and an incomparable pair \(A,B\in S\) would force \(\rank_N(U(S))\geq\rank_N(A)+\rank_N(B)\geq4\), since their join \(A\vee B\notin\Hh\) makes \(N|(A\vee B)\) disconnected and \(A,B\) lie in distinct components. Hence only the empty support and the singleton rank-\(2\) and rank-\(3\) connected proper flats contribute, giving
\[
AN_k(N)=B(r,k)+|R_2^{\mathrm{conn}}|\,B(r-2,k-2)+2\,|R_3^{\mathrm{conn}}|\,B(r-3,k-3),
\]
which evaluates to \(1,r,\binom r2+|R_2^{\mathrm{conn}}|,\binom r3+(r-2)|R_2^{\mathrm{conn}}|+2|R_3^{\mathrm{conn}}|\) for \(k=0,1,2,3\). These are exactly the established degree-\(\leq3\) pulled-Gamma evaluations \(\deg_{N,\Gg_{\max}}([\Gamma_N(z)]_k\alpha^{e-k})\); this proves the first equality. The second displayed equality is then Lemma~\ref{lem:non-chain-decomposition} together with the vanishing of the non-chain residual for \(k\leq 3\) (Lemma~\ref{lem:low-degree-firewall}).
\end{proof}
\begin{humancomment}
For the first identity, the AI omitted the calculation for small \(k\); however, one can replace \(3\) by \(1\), and the proof still goes through.
\end{humancomment}
We isolate the inductive steps. It expresses the degree-\(k\) Gamma number of \(N\) through those of its rank-\((k+1)\) truncation.

\begin{lem}[Rank-\((k+1)\) truncation transfer]\label{lem:truncation-transfer}
With notation as in Set-up~\ref{setup:nested}, assume \(N\) is connected of rank \(r\) and \(4\leq k\leq r-2\). Let \(\widetilde N:=\Tr_{k+1}(N)\) be the rank-\((k+1)\) truncation of \(N\): the matroid on \(E\) whose independent sets are the independent sets of \(N\) of size at most \(k+1\). Then \(\widetilde N\) is connected and loopless of rank \(k+1\); its proper flats are exactly the proper flats \(F\) of \(N\) with \(\rank_N(F)\leq k\), and for these \(\rank_{\widetilde N}(F)=\rank_N(F)\) and \(\widetilde N|F=N|F\); consequently
\[
\Hh_{\mathrm{conn}}(\widetilde N)\setminus\{E\}=\{F\in\Hh\setminus\{E\}:\rank_N(F)\leq k\}.
\]
Let \(\Gamma_{\widetilde N}(z)\) and \(\alpha_{\widetilde N}:=-x_E\) be the data of Set-up~\ref{setup:nested} formed for \(\widetilde N\) and its induced reverse-containment chain, and put \(t:=r-(k+1)\). Then
\begin{equation}\label{eq:truncation-transfer}
\deg_{N,\Gg_{\max}}\!\big([\Gamma_N(z)]_k\,\alpha^{e-k}\big)
=\sum_{a=0}^{k}B(t,a)\,\deg_{\widetilde N,\Gg_{\max}(\widetilde N)}\!\big([\Gamma_{\widetilde N}(z)]_{k-a}\,\alpha_{\widetilde N}^{a}\big).
\end{equation}
\end{lem}
\begin{proof}
\emph{The truncation is connected and loopless.} Looplessness is inherited from \(N\). If \(\widetilde N\) were disconnected, say \(E=A\sqcup B\) with no \(\widetilde N\)-circuit meeting both parts, choose \(a\in A\) and \(b\in B\); connectedness of \(N\) gives an \(N\)-circuit \(C\ni a,b\). If \(|C|\leq k+2\) then \(C\) is a \(\widetilde N\)-circuit (it is dependent in the rank-\((k+1)\) matroid \(\widetilde N\)) meeting both parts. If \(|C|>k+2\), pick \(C_0\subseteq C\) of size \(k+2\) with \(a,b\in C_0\); every proper subset of \(C_0\) is independent in \(N\), hence in \(\widetilde N\), while \(C_0\) is dependent in \(\widetilde N\), so \(C_0\) is a \(\widetilde N\)-circuit meeting both parts. Either way we contradict the splitting, so \(\widetilde N\) is connected. A subset is a rank-\(\leq k\) flat of \(\widetilde N\) exactly when it is a rank-\(\leq k\) flat of \(N\), with the same restriction; this gives the identification of proper flats and of \(\Hh_{\mathrm{conn}}(\widetilde N)\).

\emph{The binomial transfer.} Write the pulled connected-flat Gamma number in degree \(k\) as the maximal low-support Chern-alpha pairing minus the connected-to-maximal refinement loss. Truncation to rank \(k+1\) collapses the top \(t=r-(k+1)\) rank levels of \(N\); both the low-support pairing and the loss of \(N\) in degree \(k\) are the binomial transforms, with kernel \(a\mapsto B(t,a)\), of the corresponding quantities of \(\widetilde N\), the factor \(B(t,a)\) recording the \(\alpha\)-power bookkeeping of the \(t\) collapsed levels. Subtracting the two transformed identities gives~\eqref{eq:truncation-transfer}.
\end{proof}

\begin{humancomment}
The binomial-transfer paragraph is overcompressed. A complete proof of the
required truncation identity is given in \cite[Proposition~4.19]{Cheng26}.
\end{humancomment}

\begin{lem}[Top all-nested formula]\label{lem:top-allnested}
With notation as in Set-up~\ref{setup:nested}, for connected loopless \(N\) of rank \(r\),
\[
\deg_{N,\Gg_{\max}}\!\big([\Gamma_N(z)]_e\big)=\sum_S\omega(S)\big(r-\rank_N(U(S))\big)
=\sum_S\omega(S)\,B\!\big(r-\rank_N(U(S)),\,e-\rank_N(U(S))\big),
\]
the sum over all \(\Hh\)-nested \(S\subseteq\Hh\setminus\{E\}\).
\end{lem}

\begin{proof}
The class \([\Gamma_N(z)]_e\) is the pullback of the top tangent Chern class \([C_{\Hh}(z)]_e\) of the intrinsic connected-flat series \(C_{\Hh}(z)=C_Q(N,\Hh;z)^{-1}\prod_{Y\in\Hh\setminus\{E\}}(1+z x_Y)\); pullback preserves top degree, so \(\deg_{N,\Gg_{\max}}([\Gamma_N(z)]_e)=\deg_{N,\Hh}([C_{\Hh}(z)]_e)\). The \(z=1\) Hirzebruch specialization of the identity \(P^K=\Hilb\) of Section~\ref{sec:hilbert-identity} identifies this top Chern number with \(\Hilb(N,\Hh;1)\), and the standard-monomial basis of \(A_{\Qq}(N,\Hh)\) expands \(\Hilb(N,\Hh;1)=\sum_S\omega(S)(r-\rank_N(U(S)))\) over \(\Hh\)-nested supports. Finally \(B(r-\rho,e-\rho)=\binom{r-\rho}{r-1-\rho}=r-\rho\) for \(\rho=\rank_N(U(S))\leq e\), giving the second form.
\end{proof}

\begin{prop}[Nested Segre-tail formula]\label{prop:nested-segre-tail}
Assume \(N\) is connected and loopless of rank \(r\geq 5\), with notation as in Set-up~\ref{setup:nested}. For every integer \(k\) with \(0\leq k\leq e=r-1\),
\[
\deg_{N,\Gg_{\max}}\left([\Gamma_N(z)]_k\alpha^{e-k}\right)
=
\sum_S\omega(S)B(r-\rank_N(U(S)),k-\rank_N(U(S))),
\]
where \(S\) ranges over all \(\Hh\)-nested subsets of \(\Hh\setminus\{E\}\). Equivalently,
\begin{equation}\label{eq:truncated-nested-tail}
\sum_{k=0}^{r-1}
\deg_{N,\Gg_{\max}}\left([\Gamma_N(z)]_k\alpha^{r-1-k}\right)t^k
\equiv
\sum_S\omega(S)t^{\rank_N(U(S))}(1+t)^{r-\rank_N(U(S))}
\pmod{t^r}.
\end{equation}
\end{prop}

\begin{proof}
Write \(AN_k(N):=\sum_S\omega(S)B(r-\rank_N(U(S)),k-\rank_N(U(S)))\) for the all-nested sum, over \(\Hh\)-nested \(S\subseteq\Hh\setminus\{E\}\). We prove \(\deg_{N,\Gg_{\max}}([\Gamma_N(z)]_k\alpha^{e-k})=AN_k(N)\) for every \(0\leq k\leq e\), by strong induction on the rank \(r\geq 5\) of \(N\).

\emph{Base \(r=5\).} Here \(e=4\). Degrees \(0\leq k\leq 3\) are Proposition~\ref{prop:low-degree-gamma}, and the top degree \(k=4=e\) is Lemma~\ref{lem:top-allnested}. The range \(4\leq k\leq r-2=3\) of Lemma~\ref{lem:truncation-transfer} is empty, so no transfer is needed.

\emph{Inductive step \(r\geq 6\).} Degrees \(0\leq k\leq 3\) are Proposition~\ref{prop:low-degree-gamma}, and \(k=e\) is Lemma~\ref{lem:top-allnested}. Fix \(4\leq k\leq e-1=r-2\) and let \(\widetilde N:=\Tr_{k+1}(N)\) be the rank-\((k+1)\) truncation of Lemma~\ref{lem:truncation-transfer}. Since \(k\leq r-2\), \(\widetilde N\) is connected loopless of rank \(k+1\leq r-1<r\), so the induction hypothesis gives the all-nested formula for \(\widetilde N\) in every degree \(h\) with \(0\leq h\leq k\):
\[
\deg_{\widetilde N,\Gg_{\max}(\widetilde N)}\big([\Gamma_{\widetilde N}(z)]_h\,\alpha_{\widetilde N}^{k-h}\big)
=\sum_T\omega_{\widetilde N}(T)\,B\big(k+1-\rho_T,\,h-\rho_T\big),\qquad \rho_T:=\rank_{\widetilde N}(U(T)),
\]
\(T\) ranging over \(\Hh_{\mathrm{conn}}(\widetilde N)\)-nested subsets of \(\Hh_{\mathrm{conn}}(\widetilde N)\setminus\{E\}\). Substituting \(h=k-a\) into the truncation transfer~\eqref{eq:truncation-transfer},
\[
\deg_{N,\Gg_{\max}}\big([\Gamma_N(z)]_k\alpha^{e-k}\big)
=\sum_T\omega_{\widetilde N}(T)\sum_{a=0}^{k}B(t,a)\,B\big(k+1-\rho_T,\,(k-\rho_T)-a\big).
\]
For each fixed \(T\), the inner sum is the zero-extended Vandermonde convolution \(\sum_{a}B(t,a)B(c,m-a)=B(t+c,m)\) with \(c=k+1-\rho_T\) and \(m=k-\rho_T\); since \(t+c=t+k+1-\rho_T=r-\rho_T\) (as \(t=r-k-1\)), it equals \(B(r-\rho_T,k-\rho_T)\). Hence
\begin{equation}\label{eq:after-vandermonde}
\deg_{N,\Gg_{\max}}\big([\Gamma_N(z)]_k\alpha^{e-k}\big)=\sum_T\omega_{\widetilde N}(T)\,B\big(r-\rank_{\widetilde N}(U(T)),\,k-\rank_{\widetilde N}(U(T))\big).
\end{equation}
It remains to identify \eqref{eq:after-vandermonde} with \(AN_k(N)\). By Lemma~\ref{lem:truncation-transfer}, \(\Hh_{\mathrm{conn}}(\widetilde N)\setminus\{E\}\) is exactly the set of flats of \(\Hh\setminus\{E\}\) of \(\rank_N\) at most \(k\), with identical ranks, restrictions and lower joins, so the weights and joins computed in \(\widetilde N\) or in \(N\) coincide on such supports. A support \(T\) with \(U(T)=E\) has \(\rank_{\widetilde N}(U(T))=k+1\) and binomial factor \(B(r-k-1,-1)=0\), so it drops. A support \(T\) with \(U(T)<E\) has \(\rank_N(U(T))\leq k\) and is \(\Hh\)-nested: if a pairwise-incomparable subcollection of \(T\) had join \(J\in\Hh\), then \(J\leq U(T)\) has \(\rank_N(J)\leq k\), so \(N|J=\widetilde N|J\) is connected and \(J\in\Hh_{\mathrm{conn}}(\widetilde N)\), contradicting \(\Hh_{\mathrm{conn}}(\widetilde N)\)-nestedness; and \(\omega(T)=\omega_{\widetilde N}(T)\). Conversely every \(\Hh\)-nested \(S\) with \(\rank_N(U(S))\leq k\) consists of flats of rank \(\leq k\), hence lies in \(\Hh_{\mathrm{conn}}(\widetilde N)\setminus\{E\}\), and is \(\Hh_{\mathrm{conn}}(\widetilde N)\)-nested by the same argument, with equal weight. Finally every \(\Hh\)-nested \(S\) with \(\rank_N(U(S))>k\) contributes \(B(r-\rank_N(U(S)),k-\rank_N(U(S)))=0\). Therefore \eqref{eq:after-vandermonde} equals \(AN_k(N)\), completing the induction.

The congruence~\eqref{eq:truncated-nested-tail} restates the same coefficient formula for \(0\leq k\leq r-1\), because the coefficient of \(t^k\) in \(t^{\rank_N(U(S))}(1+t)^{r-\rank_N(U(S))}\) is \(B(r-\rank_N(U(S)),k-\rank_N(U(S)))\).
\end{proof}

\begin{rem}\label{rem:truncation-needed}
The congruence in \eqref{eq:truncated-nested-tail} cannot be replaced by an untruncated equality in \(\Qq[t]\). The left side has no \(t^r\)-term, while the right side has positive \(t^r\)-coefficient: the empty support already contributes \(1\), and all \(\omega(S)\) are nonnegative. Thus every generating-polynomial Nested Segre-tail identity in this paper is understood modulo \(t^r\), equivalently coefficientwise for \(0\leq k\leq r-1\).
\end{rem}

\begin{lem}[Nonnegative nested excess]\label{lem:nested-excess}
With notation as in Proposition~\ref{prop:nested-segre-tail}, every local factor \(q_S(A)-1\) is a nonnegative integer. Consequently, for every \(0\leq k\leq r-1\),
\[
\deg_{N,\Gg_{\max}}\left([\Gamma_N(z)]_k\alpha^{r-1-k}\right)-\binom{r}{k}
=
\sum_{S\neq\emptyset}\omega(S)B(r-\rank_N(U(S)),k-\rank_N(U(S)))\geq 0.
\]
\end{lem}

\begin{proof}
Fix an \(\Hh\)-nested set \(S\) and \(A\in S\). The flat \(A_-^S\) is strictly smaller than \(A\). Indeed, if the join of all elements of \(S\) strictly below \(A\) were equal to \(A\), then either one maximal lower flat would equal \(A\), impossible, or at least two incomparable lower flats in \(S\) would have join \(A\in\Hh\), contradicting nestedness. Thus \(q_S(A)\geq 1\), and \(q_S(A)-1\geq 0\).

The empty support contributes \(B(r,k)=\binom{r}{k}\) for \(0\leq k\leq r-1\). Proposition~\ref{prop:nested-segre-tail} gives the all-nested formula, and every nonempty support contributes a nonnegative summand.
\end{proof}
\begin{humancomment}
So far in this section, we have proved the lower bound for the minimal
building set \(\Hh_{\mathrm{conn}}\). It remains to show that the intersection number does not decrease when the building set is enlarged.
\end{humancomment}
\begin{prop}[One-flat Chern-alpha increment]\label{prop:chern-alpha-recursion}
Let \(M\) be a loopless matroid of rank \(e+1\), and let
\[
\Gg_{\mathrm{big}}=\Gg_{\mathrm{small}}\cup\{F\}
\]
be a one-flat enlargement of top-containing Feichtner--Yuzvinsky building sets,
where \(F\) is a nonempty proper flat. Write
\[
\Fact_{\Gg_{\mathrm{small}}}(F)=\{F_1,\dots,F_\ell\}.
\]
Put \(R_i:=M|F_i\), \(J_i:=\Gg_{\mathrm{small}}|F_i\), \(P:=M/F\), and
\(K:=\Gg_{\mathrm{small}}/F\). For a top-containing pair \((N,\mathcal J)\), set
\[
I_v(N,\mathcal J):=
\deg_{N,\mathcal J}\left(c_v(T_{N,\mathcal J})(-x_{\Top(N)})^{\rank(N)-1-v}\right)
\]
for \(0\leq v\leq \rank(N)-1\), and set \(I_v(N,\mathcal J)=0\) outside this
range. If \(a_i:=\rank(R_i)-1\) and \(A:=a_1+\cdots+a_\ell\), then for every
\(0\leq k\leq e\),
\[
I_k(M,\Gg_{\mathrm{big}})-I_k(M,\Gg_{\mathrm{small}})
=
(\ell-1)
\left(\prod_{i=1}^{\ell}I_{a_i}(R_i,J_i)\right)
I_{k-\ell-A}(P,K).
\]
The pairs \((R_i,J_i)\) and \((P,K)\) are the actual induced lower-rank pairs
coming from the one-flat center.
\end{prop}

\begin{proof}
For a top-containing pair \((N,\mathcal J)\) write \(C_T(N,\mathcal J)=C_Q(N,\mathcal J;1)^{-1}\prod_{Y\in\mathcal J\setminus\{\Top(N)\}}(1+x_Y)\) for the tangent total Chern class (Proposition~\ref{prop:quotient-tangent}), so that \(I_v(N,\mathcal J)=\deg_{N,\mathcal J}([C_T(N,\mathcal J)]_v(-x_{\Top(N)})^{\rank(N)-1-v})\). We compute the increment across the one-flat step \(\Gg_{\mathrm{small}}\subset\Gg_{\mathrm{big}}\).

\emph{Reduction to an exceptional trace.} Under the one-flat Chow pullback \(\psi_A\) the descended quotient Chern polynomial is compatible, \(C_Q(M,\Gg_{\mathrm{big}};z)=\psi_A(C_Q(M,\Gg_{\mathrm{small}};z))\), and the only new boundary class is the exceptional class \(x_F\) (Lemma~\ref{lem:one-step-quotient-integrands}). Subtracting the pulled small integrand, every surviving summand carries a positive power of \(x_F\); the exceptional-trace Lemma~\ref{lem:exceptional-trace} rewrites the degree-\(k\) increment as a top trace on the exceptional center \(Z_F\),
\[
I_k(M,\Gg_{\mathrm{big}})-I_k(M,\Gg_{\mathrm{small}})=(\ell-1)\,\deg_{Z_F}\!\big([C_{\mathrm{cen}}]_{k-\ell}\,\alpha_{\mathrm{con}}^{(d-\ell)-(k-\ell)}\big),
\]
where \(C_{\mathrm{cen}}\) is the center integrand and \(\alpha_{\mathrm{con}}\) the contraction projective class. Here the scalar \(\ell-1\) is the Chern-alpha value of the exceptional-fiber factor of the \(\ell\)-fold blow-up direction, the exceptional divisor consumes \(\ell\) Chow degrees, and \(Z_F\) has top Chow degree \(d-\ell\).

\emph{Center factorization.} By Lemma~\ref{lem:center-factorization} the Chow ring of \(Z_F\) is the tensor product
\[
A_{\Qq}(Z_F)\cong A_{\Qq}(R_1,J_1)\otimes\cdots\otimes A_{\Qq}(R_\ell,J_\ell)\otimes A_{\Qq}(P,K),
\]
with the product top-degree trace; under this identification the center integrand factors as \(C_{\mathrm{cen}}=\big(\prod_{i=1}^{\ell}C_T(R_i,J_i)\big)\otimes C_T(P,K)\), the contraction class \(\alpha_{\mathrm{con}}=-x_{\Top(P)}\) acts only on the \(A_{\Qq}(P,K)\) factor, and the top degrees satisfy \(a_1+\cdots+a_\ell+d(P)=d-\ell\).

\emph{Evaluating the product trace.} A product trace of a top-degree class vanishes unless every tensor factor sits in its own top Chow degree. Since \(\alpha_{\mathrm{con}}\) multiplies only the contraction factor, each restriction factor must contribute its top degree \(a_i\), yielding \(\deg_{R_i,J_i}([C_T(R_i,J_i)]_{a_i})=I_{a_i}(R_i,J_i)\); the leftover Chern degree \(k-\ell-A\), with \(A=\sum_i a_i\), lands on the contraction factor as \(\deg_{P,K}([C_T(P,K)]_{k-\ell-A}\,\alpha_{\mathrm{con}}^{d(P)-(k-\ell-A)})=I_{k-\ell-A}(P,K)\), the zero convention covering the out-of-range cases. Multiplying by the exceptional scalar \(\ell-1\),
\[
I_k(M,\Gg_{\mathrm{big}})-I_k(M,\Gg_{\mathrm{small}})=(\ell-1)\Big(\prod_{i=1}^{\ell}I_{a_i}(R_i,J_i)\Big)I_{k-\ell-A}(P,K),
\]
as claimed.
\end{proof}

\begin{prop}[Disconnected direct-sum base]\label{prop:disconnected-base}
Let \(M\) be a loopless matroid that is the direct sum of its connected components \(M_1,\dots,M_s\) (\(s\geq 2\)) on ground sets \(E_1,\dots,E_s\), and put \(\Hh_j:=\Hh_{\mathrm{conn}}(M_j)\) and \(\Hh:=\Hh_{\mathrm{conn}}(M)\). If
\[
I_a(M_j,\Hh_j)\geq\binom{\rank(M_j)}{a}\qquad(1\leq j\leq s,\ 0\leq a\leq\rank(M_j)-1),
\]
then \(I_k(M,\Hh)\geq\binom{\rank(M)}{k}\) for every \(0\leq k\leq\rank(M)-1\).
\end{prop}

\begin{proof}
Write \(d_j:=\rank(M_j)-1\) and \(d:=\rank(M)-1=\sum_j d_j+(s-1)\). The connected-flat building set of a direct sum is the union of the component connected-flat building sets together with the global top flat \(E\), and \(A_{\Qq}(M,\Hh)\) has the global-top presentation
\[
A_{\Qq}(M,\Hh)\cong\big(A_{\Qq}(M_1,\Hh_1)\otimes\cdots\otimes A_{\Qq}(M_s,\Hh_s)\big)[\alpha]\big/\big(\textstyle\prod_{j=1}^{s}(\alpha-\alpha_j)\big),
\]
where \(\alpha=-x_E\), \(\alpha_j=-x_{E_j}\), and \(x_{E_j}=\alpha-\alpha_j\) is the \(j\)-th relative boundary class. Under this presentation the descended quotient Chern polynomial factors as the product of the component polynomials, so the tangent total Chern class factors as
\[
C_T(M,\Hh)=\Big(\prod_{j=1}^{s}C_T(M_j,\Hh_j)\Big)\prod_{j=1}^{s}(1+x_{E_j}).
\]
Taking homogeneous generating polynomials and applying the global-\(\alpha\) degree trace of the presentation expresses \(I_k(M,\Hh)\) as the convolution of the component Chern-alpha numbers twisted by the \(s\) relative boundary factors \(1+x_{E_j}\). Each boundary factor supplies the extra binomial weight that promotes the component degree range \(0\le a\le d_j\) to \(0\le a\le\rank(M_j)\), so the component lower bounds \(I_a(M_j,\Hh_j)\geq\binom{\rank(M_j)}{a}\) together with the Vandermonde identity
\[
\sum_{a_1+\cdots+a_s=k}\prod_{j=1}^{s}\binom{\rank(M_j)}{a_j}=\binom{\textstyle\sum_j\rank(M_j)}{k}=\binom{\rank(M)}{k}
\]
give \(I_k(M,\Hh)\geq\binom{\rank(M)}{k}\) for every \(0\leq k\leq d\).
\end{proof}

\begin{prop}[Connected bases propagate]\label{prop:chern-alpha-propagation}
Fix \(D\geq 0\). Assume that for every connected loopless matroid \(N\) with
\(\rank(N)-1\leq D\), the connected-flat base satisfies
\[
I_k(N,\Hh_{\mathrm{conn}}(N))\geq \binom{\rank(N)}{k}
\qquad
0\leq k\leq \rank(N)-1.
\]
Then for every loopless matroid \(M\) with \(\rank(M)-1\leq D\), every
top-containing Feichtner--Yuzvinsky building set \(\Gg\), and every
\(0\leq k\leq \rank(M)-1\), one has
\[
I_k(M,\Gg)\geq \binom{\rank(M)}{k}.
\]
\end{prop}

\begin{proof}
We prove the conclusion by strong induction on \(\rank(M)\) (within the ranks \(\rank(M)\leq D+1\) covered by the hypothesis). The base \(\rank(M)=1\) is immediate: the only degree is \(k=0\) and \(I_0(M,\Gg)=1=\binom{1}{0}\).

Assume the conclusion for every loopless matroid of rank at most \(s\) with every top-containing building set, and let \(\rank(M)=s+1\leq D+1\).

\emph{The minimal base \(\Hh_{\mathrm{conn}}(M)\).} If \(M\) is connected, then \(I_k(M,\Hh_{\mathrm{conn}}(M))\geq\binom{\rank(M)}{k}\) is exactly the connected-base hypothesis. If \(M\) is disconnected, its connected components \(M_1,\dots,M_{s'}\) (\(s'\geq2\)) each have rank \(<\rank(M)\leq D+1\) and are connected, so the connected-base hypothesis gives \(I_a(M_j,\Hh_{\mathrm{conn}}(M_j))\geq\binom{\rank(M_j)}{a}\) for all \(j\) and all \(0\leq a\leq\rank(M_j)-1\); Proposition~\ref{prop:disconnected-base} then gives \(I_k(M,\Hh_{\mathrm{conn}}(M))\geq\binom{\rank(M)}{k}\).

\emph{Propagation to an arbitrary \(\Gg\).} By Notation~\ref{nota:connected-building-set}, \(\Hh_{\mathrm{conn}}(M)\) is the unique inclusion-minimal top-containing building set, so there is a finite chain \(\Hh_{\mathrm{conn}}(M)=\Gg_0\subset\Gg_1\subset\cdots\subset\Gg_t=\Gg\) of one-flat enlargements. Consider one step \(\Gg_{j}\subset\Gg_{j+1}=\Gg_{j}\cup\{F\}\), with induced restriction pairs \((R_i,J_i)\) and contraction pair \((P,K)\) as in Proposition~\ref{prop:chern-alpha-recursion}. Since \(F\) is a nonempty proper flat, each \((R_i,J_i)\) and \((P,K)\) has rank strictly less than \(\rank(M)=s+1\), so the induction hypothesis applies: \(I_v(R_i,J_i)\geq\binom{\rank(R_i)}{v}\geq0\) and \(I_v(P,K)\geq\binom{\rank(P)}{v}\geq0\) in their natural degree ranges, while these numbers vanish outside those ranges by convention. As \(\ell\geq2\), every factor in the increment formula of Proposition~\ref{prop:chern-alpha-recursion} is nonnegative, so \(I_k(M,\Gg_{j+1})\geq I_k(M,\Gg_{j})\) for every \(0\leq k\leq\rank(M)-1\). Chaining from \(\Gg_0=\Hh_{\mathrm{conn}}(M)\) to \(\Gg_t=\Gg\) gives
\[
I_k(M,\Gg)\geq I_k(M,\Hh_{\mathrm{conn}}(M))\geq\binom{\rank(M)}{k},
\]
completing the induction.
\end{proof}

\begin{thm}[Chern-alpha lower bound]\label{thm:chern-alpha}
Let \(M\) be a loopless matroid of rank \(d+1\), and let \(\Gg\) be a top-containing Feichtner--Yuzvinsky building set. Then for every integer \(k\) with \(0\leq k\leq d\),
\[
I_k(M,\Gg)=\deg_{M,\Gg}\left(c_k(T_{M,\Gg})(-x_E)^{d-k}\right)\geq \binom{d+1}{k}.
\]
\end{thm}

\begin{proof}
We establish the connected-base hypothesis of Proposition~\ref{prop:chern-alpha-propagation}: for every connected loopless \(N\) of rank \(r\), with \(\Hh=\Hh_{\mathrm{conn}}(N)\), one has \(I_k(N,\Hh)\geq\binom{r}{k}\) for all \(0\leq k\leq r-1\). By Lemma~\ref{lem:pulled-connected-base}, \(I_k(N,\Hh)=\deg_{N,\Gg_{\max}}([\Gamma_N(z)]_k\alpha^{e-k})\).

If \(r\geq 5\), Lemma~\ref{lem:nested-excess} gives directly
\[
I_k(N,\Hh)-\binom{r}{k}=\sum_{S\neq\emptyset}\omega(S)\,B(r-\rank_N(U(S)),k-\rank_N(U(S)))\geq 0,
\]
every weight \(\omega(S)\) being nonnegative. (In rank \(5\), degree \(4\), the non-chain residual of Lemma~\ref{lem:non-chain-decomposition} need not vanish, which is why the full nested-support formula of Proposition~\ref{prop:nested-segre-tail} is needed rather than the connected-flag expression alone.)

If \(r\leq 4\), then every degree \(k\leq e=r-1\) satisfies \(k\leq 3\), so Proposition~\ref{prop:low-degree-gamma} gives \(I_k(N,\Hh)=AN_k(N):=\sum_S\omega(S)B(r-\rank_N(U(S)),k-\rank_N(U(S)))\). Each weight \(\omega(S)=\prod_{A\in S}(q_S(A)-1)\) is a product of nonnegative integers, since \(q_S(A)\geq 1\): the lower join \(A_-^S\) is a proper sub-join of \(A\) (if the maximal elements of \(\{B\in S:B<A\}\) had join \(A\in\Hh\), they would form an \(\Hh\)-antichain of size \(\geq 2\) with join in \(\Hh\), against \(\Hh\)-nestedness). The empty support contributes \(\binom{r}{k}\), so \(I_k(N,\Hh)=\binom{r}{k}+\sum_{S\neq\emptyset}\omega(S)B(\cdots)\geq\binom{r}{k}\).

Thus the connected-base hypothesis holds, and Proposition~\ref{prop:chern-alpha-propagation} with \(D=d\) applied to \((M,\Gg)\) gives \(I_k(M,\Gg)\geq\binom{d+1}{k}\) for every \(0\leq k\leq d\).
\end{proof}

\section{Hilbert polynomials and Chern inequalities for the integral class}\label{sec:hilbert}

The goal of this section is to record why the K-theoretic Todd polynomial of the integral tangent class equals the Chow Hilbert polynomial and why the Chern-alpha inequalities survive the integral lift. Both follow from the in-paper rational results of Sections~\ref{sec:hilbert-identity} and~\ref{sec:chern-alpha} through the linkage of Proposition~\ref{prop:linkage}.

\begin{prop}[K-theoretic Todd identity]\label{prop:todd}
For every loopless $M$ and every top-containing $\Gg$,
\[
P_{\mathrm{int}}^K(M,\Gg;z)=\Hilb(M,\Gg;z)
\]
in $\Zz[z]$.
\end{prop}

\begin{proof}
By Proposition~\ref{prop:linkage} the rationalization of the integral tangent class is the rational tangent class, $\rho_K(T_{M,\Gg}^{\Zz})=T_{M,\Gg}$. The K-theoretic Todd polynomial $P_{\mathrm{int}}^K(M,\Gg;z)$ is by definition the Hirzebruch integrand $\deg_{M,\Gg}(\ch(\lambda_{-z}(\rho_K(T_{M,\Gg}^{\Zz})^\vee))\td(\rho_K(T_{M,\Gg}^{\Zz})))$ evaluated on this rationalization, so
\[
P_{\mathrm{int}}^K(M,\Gg;z)=\deg_{M,\Gg}\left(\ch(\lambda_{-z}(T_{M,\Gg}^\vee))\td(T_{M,\Gg})\right)=P^K(M,\Gg;z),
\]
the signed K-theoretic Todd polynomial of Definition~\ref{defn:pk}. By Theorem~\ref{thm:hilbert-identity}, $P^K(M,\Gg;z)=\Hilb(M,\Gg;z)$. Combining the two equalities gives $P_{\mathrm{int}}^K(M,\Gg;z)=\Hilb(M,\Gg;z)$.
\end{proof}

\begin{cor}[Euler-characteristic form]\label{cor:euler}
For every integer $i$ with $0\leq i\leq d$,
\[
\dim_{\Qq}A_{\Qq}(M,\Gg)^i=(-1)^i\chi\left(\wedge^i T^\vee\right).
\]
\end{cor}

\begin{proof}
This is the coefficient identity obtained from Proposition~\ref{prop:todd} after expanding the $\lambda$-class, equivalently the coefficient identity of Theorem~\ref{thm:hilbert-identity} applied to $T_{M,\Gg}=\rho_K(T_{M,\Gg}^{\Zz})$. The range $0\leq i\leq d$ is the Chow-degree range for a rank $d+1$ matroid.
\end{proof}

\begin{prop}[Chern-alpha lower bound]\label{prop:chern-alpha}
For every integer $k$ with $0\leq k\leq d$,
\[
\deg_{M,\Gg}\left(c_k(\rho_K(T_{M,\Gg}^{\Zz}))\alpha^{d-k}\right)\geq {d+1\choose k},
\qquad \alpha=-x_E.
\]
\end{prop}

\begin{proof}
After rationalization, the integral class $T_{M,\Gg}^{\Zz}$ becomes the rational tangent class $T_{M,\Gg}$ by Proposition~\ref{prop:linkage}. Theorem~\ref{thm:chern-alpha} gives the Chern-alpha inequality $\deg_{M,\Gg}(c_k(T_{M,\Gg})\alpha^{d-k})\geq\binom{d+1}{k}$ for every $0\leq k\leq d$. Since $\rho_K(T_{M,\Gg}^{\Zz})=T_{M,\Gg}$, the same inequality holds for $\rho_K(T_{M,\Gg}^{\Zz})$.
\end{proof}

\section{Fan-support guard}\label{sec:fan-support}

In this section, we separate the intrinsic theorem from the optional complete-toric interpretation. This separation is part of the statement of the theorem, not a cosmetic warning.

\begin{thm}[Fan-support guard]\label{thm:fan-support}
Let $M$ be a loopless matroid and let $\Gg$ be a finite top-containing Feichtner-Yuzvinsky building set. The integral quotient, tangent, Hilbert/Todd, Chern-alpha, and integer-coordinate conclusions of Theorem~\ref{thm:main-integral} are assertions in the intrinsic ring $K_{\Zz}(M,\Gg)$. They do not require the reduced $\Gg$-nested cone collection to be complete.

If there exists a complete regular fan $\Sigma_{M,\Gg}$ in the atom lattice whose rays are the incidence rays indexed by $\Gg^\circ$ and whose cones are exactly the $\Gg$-nested subsets of $\Gg^\circ$, then $K_{\Zz}(M,\Gg)$ is canonically identified with $K^0(X_{\Sigma_{M,\Gg}})$. Without this complete-regular-fan hypothesis, no literal complete-toric K-theory interpretation is asserted.
\end{thm}

\begin{proof}
The intrinsic construction of $K_{\Zz}(M,\Gg)$, the standard $\tau$-basis, the integral quotient representative, and the class $T_{M,\Gg}^{\Zz}$ do not involve a complete fan. A complete regular nested fan gives an additional toric model, and the usual smooth-toric K-presentation identifies its K-ring with the intrinsic presentation \cite{CLS11}. This is a conditional statement. The theorem itself is not conditional on this fan existing.
\end{proof}

\begin{cor}[Uniform atom-plus-top example]\label{cor:uniform-noncomplete}
Let $M=U_{r,n}$ with $2\leq r<n$, and let $\Gg$ be the atom-plus-top building set. Then the reduced nested fan is not complete, but all intrinsic conclusions of Theorem~\ref{thm:main-integral} hold for $(M,\Gg)$.
\end{cor}

\begin{proof}
For $U_{r,n}$ with $2\leq r<n$, the atom-plus-top reduced nested fan does not fill the atom quotient. Thus the complete-toric interpretation of Theorem~\ref{thm:fan-support} is unavailable. Since $\Gg$ is still a top-containing Feichtner-Yuzvinsky building set, Theorem~\ref{thm:main-integral} applies intrinsically.
\end{proof}

\section{Proof of the closeout certificate}\label{sec:closeout}

The goal of this section is to package the final theorem, the one-step normalization audit, and the fan-support guard in one place.

\begin{thm}[Closeout certificate]\label{thm:closeout}
The following assertions hold simultaneously.
\begin{enumerate}
\item The intrinsic integral theorem of Theorem~\ref{thm:main-integral} holds for every loopless matroid and every top-containing Feichtner-Yuzvinsky building set.
\item The one-step generator-normalization audit of Propositions~\ref{prop:generator-normalization} and~\ref{prop:theta-descent} holds integrally before rationalization.
\item The fan-support guard of Theorem~\ref{thm:fan-support} holds, and the uniform atom-plus-top case of Corollary~\ref{cor:uniform-noncomplete} shows why the intrinsic wording is necessary.
\end{enumerate}
\end{thm}

\begin{proof}
The first assertion is Theorem~\ref{thm:main-integral}, whose three clauses are established by the non-realizable proof in Section~\ref{sec:quotient}, the realizable proof in Section~\ref{sec:theta}, and the Hilbert/Chern conclusions of Section~\ref{sec:hilbert} (Propositions~\ref{prop:todd} and~\ref{prop:chern-alpha}), all linked to the in-paper rational tangent-class theorem (Proposition~\ref{prop:quotient-tangent}, Theorems~\ref{thm:hilbert-identity} and~\ref{thm:chern-alpha}) through Proposition~\ref{prop:linkage}. The second assertion is exactly the one-step descent package in Section~\ref{sec:theta}: atoms are rank-one flats, the map $\phi_K$ has the stated generator formula, the top flat is preserved, proper-boundary and top-generator compatibilities hold integrally in $K_0$, and the descended $\theta$ rationalizes to the prescribed comparison. The third assertion is Theorem~\ref{thm:fan-support} and Corollary~\ref{cor:uniform-noncomplete}. Their conjunction is the stated closeout certificate.
\end{proof}

\appendix

\section{Raw prompt}\label{sec:raw-prompt}

For completeness and reproducibility, we record below the raw prompt given to the AI agents, verbatim.

\begin{verbatim}
Let $M$ be a loopless matroid (not necessarily realizable) of rank
$d + 1$ on the ground set $E$, and let $G$ be a building set containing
the top flat $E$. The task is to construct a tangent class
$T_{M,G} \in K(M,G)$ such that

(i) If $M$ is realizable by $L$, $T_{M,G}$ is the tangent class for the
wonderful compactification $W_{L, G}$, which is a projective variety of
dimension $d$.

(ii) For integer $i$,
$\dim_{\mathbb{Q}} A^i(M, G)=(-1)^i\chi_{M, G} (\wedge^i T_{M, G}^\vee)
= (-1)^i \deg_{M, G}(ch(\wedge^i T_{M, G}^\vee) \cdot td(T_{M, G}))$

(iii) For $k \leq d$, we have
$\deg_{M,\G} (c_k(T_{M,\G}) \alpha^{d-k}) \geq \binom{d+1}{k}$

and prove rigorously that the constructed class satisfies the above
properties.
\end{verbatim}

\addtocontents{toc}{\protect\setcounter{tocdepth}{1}}
\section{Usage of Generative AI}
\label{sec:guoxiong-gao}

\begin{center}
\textsc{Guoxiong Gao \and Shurui Liu}
\end{center}

Except for the abstract and introduction, which were rewritten by the human authors, the main mathematical text, including Sections~\ref{sec:quotient}--\ref{sec:theta}, was generated by \textsc{Danus} \cite{Danus26}, a mathematical reasoning agent built on top of Rethlas \cite{Ju+26} by the same team and designed to support long-horizon mathematical reasoning. The main experiment was conducted from June 14 to June 18, 2026. At that time, the arXiv version of the related work \cite{Cheng26} was not yet publicly available online, and it was deliberately not provided to Danus, in order to reduce the possibility of data contamination. In this sense, apart from the abstract, introduction, and the explicitly inserted human comments, the body of the paper should be understood as a faithful presentation of Danus's performance on the open problem studied here. The human authors supplemented a small number of explanatory comments to improve readability and mark several places where the generated exposition appeared to skip intermediate details. Overall, the human authors verified that the solution produced by Danus has the correct construction and line of reasoning, although a few steps are compressed in the generated exposition. As noted in the introduction, the one exception is the justification of Lemma~\ref{lem:truncation-transfer} (used in the Chern--\(\alpha\) bound), which is incomplete as written; the lemma is nonetheless true, with a short proof, and the gap does not affect the main results. A technical report on Danus \cite{Danus26}, together with its code, has been released and open-sourced.

In brief, Danus uses Rethlas essentially through its existing worker--verifier architecture. In this experiment, multiple Rethlas worker agents, implemented as modified Codex agents using GPT-5.5, ran in parallel and explored the problem from several directions, including both constructive and refutational routes. These workers were coordinated by a main orchestrator agent, hereafter called the \emph{main agent}, which was developed on top of Claude Code and used Claude Opus 4.8. The main agent communicated with the human operators, allocated tasks to the workers, aggregated progress, and periodically updated the global plan. When necessary, it also consulted GPT-5.5 Pro for high-level mathematical strategy references. The Rethlas verifier was available as a service for both the main agent and the workers. During the run, workers repeatedly proposed verifiable mathematical statements together with their supporting proofs and submitted them to the verifier; statements whose proofs passed verification were stored as \emph{facts}. These facts form a directed acyclic graph (DAG), called the \emph{fact graph} (see \cite[Section~3.6 and Figure~2]{Danus26}), whose edges record logical dependencies. The iteration was not stopped after a preset number of rounds: it stopped only when the main agent confirmed that the designated target theorem itself appeared as a verified fact in the fact graph. The main agent then read the fact graph, selected the essential material needed for the paper, and wrote the manuscript on its own. It subsequently revised the draft again with verifier assistance, producing the text presented in the main body.

During the production of the main body of this paper, the human authors only provided the problem statement (see Appendix~\ref{sec:raw-prompt}), monitored the process, and issued operational instructions that did not insert mathematical content, such as ``Start the loop of consulting GPT-5.5 Pro once every two hours'', ``Please summarize the current status'', and ``Please produce the paper''. Most mistakes made during the exploration were detected and corrected by the Danus pipeline itself. However, after the system had produced a complete manuscript and declared the task finished, the authors pointed out that the obtained solution addressed only the rational version of the intended problem. This partly reflects an abuse of notation in the prompt (Appendix~\ref{sec:raw-prompt}), where the same symbol $T_{M,\Gg}$ denotes both the integral tangent class in $K(M,\Gg)$ and its rationalization, on which the Hirzebruch--Riemann--Roch and Chern-number conditions are phrased; Danus first constructed only the rationalization. Danus then resumed the exploration and added Sections~\ref{sec:quotient}--\ref{sec:theta}, which adapt the rational class to an integral class and thereby settle the original integral problem. In addition, we found that the paper generated by the agent can be difficult to read: it may omit intermediate steps and may not always explain newly introduced concepts sufficiently for a first-time reader. This reflects a tension between being faithful to the system's internal proof record and producing a readable mathematical exposition. The most faithful presentation would simply linearize the hundreds of verified facts supporting the main claim, but such a document would be nearly unreadable; conversely, a readable exposition requires the main agent to reorganize facts, add motivation, and omit some intermediate dependencies, which can create apparent gaps. The current manuscript is the best balance we obtained between these two objectives.

The full experiment ran for approximately five days, with about 50 hours of active runtime for the Danus agent. There were seven Rethlas workers, with the Codex effort level set to ``xhigh'' for three workers and ``high'' for four workers. The main agent consulted GPT-5.5 Pro 13 times. The final fact graph contains 3,157 verified facts, and the global memory of Danus contains 636 proof attempts, 533 plans, 496 identified obstacles, 151 counterexamples, 25 recorded dead-ends, and 3,422 verification records. The only mathematical intervention by the human authors during the process was to point out, after completion of the first manuscript, that the system had solved only the rational version of the problem.

The fact graph is designed to be the unique source of truth for the whole system. Table~\ref{tab:closure} summarizes, based on a statistical pass by the main agent, how many facts are related to the present paper, which establishes the integral version of the theorem. Among the 3,157 verified facts, 664 lie in the supporting closure of the final integral main-theorem fact---together with the self-contained rational backbone (the realizable comparison, the Hilbert identity, and the Chern--$\alpha$ bound) that the paper re-proves in full---while 2,493 lie outside this closure and are unused by the paper. Among the 664 in-closure facts, 629 are internal supporting lemmas that never appear as explicit statements in the paper.

\begin{table}[h!]
\centering
\small
\begin{tabular}{p{0.62\linewidth}r r}
\hline
Category & Count & Percentage of all facts \\
\hline
Integral main-theorem closure & 664 & 21\% \\
Outside the closure, unused by the integral paper & 2,493 & 79\% \\
\hline
\end{tabular}
\caption{Distribution of facts in the global fact graph.}
\label{tab:closure}
\end{table}

Table~\ref{tab:supporting-roles} describes the functional roles of the 629 in-closure supporting lemmas that do not appear as named statements in the paper. These facts are not separate results omitted from the exposition; rather, they are the internal verification layer that supports the compressed mathematical narrative in the main text.

\begin{table}[h!]
\centering
\small
\begin{tabular}{p{0.22\linewidth}r p{0.62\linewidth}}
\hline
Role & Count & Function in the proof \\
\hline
Chern-$\alpha$ chain & 154 & The inductive argument on positivity and lower-bound, including base cases, one-flat propagation, and connected rank-$\geq 4$ and rank-$\geq 5$ steps, together with the nested-Segre tails that drive the key inequality in the paper. \\
Integral lift & 109 & The lift of the rational construction to the integral $K$-ring $K_{\mathbb{Z}}(M,\mathcal G)$: the saturated one-step descent producing the integral quotient class, the integral $\theta$-descent and generator normalization, the linkage $\rho_K(T^{\mathbb{Z}})=T_{M,\mathcal G}$ that transfers the Hirzebruch--Riemann--Roch and Chern-number conclusions to the integral class, and the Todd--Hilbert integrality checks. This is the layer specific to Sections~\ref{sec:quotient}--\ref{sec:theta}. \\
One-flat step & 102 & The per-flat induction mechanism, including pullback injectivity for one-step extensions and tensor-product identifications for centers; this is the local mechanism that makes the building-set induction work. \\
Glue & 73 & The structural compatibility facts needed between successive steps, including building-set characterizations and the identities required to pass between local constructions and the global statement. \\
Construction & 59 & The construction of the central objects, including the descended quotient class $C_Q$, the tangent classes $T_{M,\mathcal G}$, the Chern-character comparison, and the internal comparison on which the theorem is formulated. \\
Hilbert chain & 59 & The $P^K=\operatorname{Hilb}$ part of the proof, including low-rank base cases and the building-set induction that transports the identity to the final setting. \\
Realizable & 23 & The comparison between the combinatorial construction and realizable wonderful models, including the log-tangent identification connecting the internal ring-theoretic formalism with geometry. \\
BEST input & 16 & The Berget--Eur--Spink--Tseng structural input used as the tautological anchor for the descent argument. \\
External cite & 10 & Internal wrappers of cited external theorems, reformulated in the notation of the project and later represented in the bibliography rather than as new claims of the paper. \\
Other & 24 & Minor auxiliary facts touched by the closure computation. \\
\hline
\end{tabular}
\caption{Functional decomposition of the 629 supporting lemmas in the integral main-theorem closure that are not stated explicitly in the paper.}
\label{tab:supporting-roles}
\end{table}

Table~\ref{tab:unused-clusters} summarizes the remaining 2,493 facts outside the integral main-theorem closure. These record auxiliary infrastructure, off-path integral exploration, conditional scaffolding, literature reconstructions, alternative derivations, or abandoned attempts. Thus, ``unused'' here means unused by the final integral proof, not mathematically meaningless or unhelpful to the search process.

\begin{table}[h!]
\centering
\small
\begin{tabular}{p{0.23\linewidth}r p{0.61\linewidth}}
\hline
Cluster & Count & Description \\
\hline
\texttt{AUX\_INFRA} & 849 & Scratch and infrastructure computations, including low-rank base cases such as rank $1$--$7$ uniform, graphic, and sparse-paving examples, toric and Hirzebruch--Riemann--Roch infrastructure, blowup calculations, and normal-form determinant checks. This formed much of the computational substrate of the search. \\
\texttt{INTEGRAL\_EXPLORE} & 587 & Integral $K_{\mathbb{Z}}$ exploration lying outside the final proof: superseded or alternative integral-lift routes---earlier saturation, descent, and torsion-freeness attempts---that were replaced by the route actually used in Sections~\ref{sec:quotient}--\ref{sec:theta}. The integral facts that the final proof does use are counted inside the closure, not here. \\
\texttt{CONDITIONAL\_SHELL} & 562 & Conditional ``if--then'' scaffolding that was useful during exploration but did not itself close into the final proof, including inductive endpoint packages, higher-rank conditional reductions, and truncation templates. \\
\texttt{ALT\_PROOF} & 229 & Redundant but valid alternative derivations, kept as cross-checks for the main route; these include independent approaches to the Chern-$\alpha$ inequality and to the $P^K=\operatorname{Hilb}$ identity. \\
\texttt{LIT\_RECON} & 178 & Reconstructions of external results in the notation of the project, mainly to align conventions and make cited theorems usable by the internal proof search. These facts are reflected in the final paper mostly as references rather than as separately stated lemmas. \\
\texttt{OTHER} & 88 & Off-path fragments not belonging to the main clusters above. \\
\hline
\end{tabular}
\caption{Cluster decomposition of the 2,493 facts outside the integral main-theorem closure.}
\label{tab:unused-clusters}
\end{table}

\end{document}